\documentclass[11pt,a4paper]{article}

\usepackage[dvipsnames]{xcolor}
\usepackage[all,cmtip]{xy}
\entrymodifiers={+!!<0pt,\fontdimen22\textfont2>}

\usepackage{fouriernc}
\usepackage[english]{babel}         
\usepackage{verbatim}               
\usepackage[T1]{fontenc}

\usepackage{amsmath}
\usepackage{amsfonts}
\usepackage{amssymb}
\usepackage{mathrsfs}    
\usepackage{amsthm}

\usepackage{todonotes}

\usepackage[colorlinks=true,linkcolor=MidnightBlue,citecolor=OliveGreen]{hyperref}           

\swapnumbers

\theoremstyle{plain}
\newtheorem*{thrm}{Theorem}
\newtheorem{thm}{Theorem}[section]

\newtheorem{prop}[thm]{Proposition}
\newtheorem{cor}[thm]{Corollary}

\newtheorem{lemma}[thm]{Lemma}
\newtheorem*{conj}{Conjecture}

\theoremstyle{definition}

\newtheorem{defn}[thm]{Definition}
\newtheorem{defns}[thm]{Definitions}

\theoremstyle{remark}
\newtheorem{remark}[thm]{Aside}
\newtheorem{eg}[thm]{Example}

\newcommand{\proofof}[1]{\end{#1}\begin{proof}}

\makeatletter
\renewcommand\section{\@startsection {section}{1}{\z@}%
  {-3.5ex \@plus -1ex \@minus -.2ex}{2.3ex \@plus.2ex}%
  {\normalfont\large\bfseries}}
\renewcommand\subsection{\@startsection{subsection}{2}{\z@}%
  {-3.25ex\@plus -1ex \@minus -.2ex}{1.5ex \@plus .2ex}%
  {\normalfont\bfseries}}

\newcommand{\lie}[1]{\mathfrak{#1}}
\newcommand{\sh}[1]{\mathcal{#1}}
\newcommand{\N}{{\mathbb N}}
\newcommand{\Z}{{\mathbb Z}}
\newcommand{\Q}{{\mathbb Q}}
\newcommand{\R}{{\mathbb R}}
\newcommand{\C}{{\mathbb C}}
\newcommand{\F}{{\mathbb F}}

\newcommand{\G}{{\mathbb G}}
\renewcommand{\P}{{\mathbb P}}
\newcommand{\A}{{\mathbb A}}

\newcommand{\tens}{\mathbin{\otimes}}

\newcommand{\Hom}{\mathrm{Hom}}

\newcommand{\GL}{\mathrm{GL}}

\newcommand{\SL}{\mathrm{SL}}
\newcommand{\Aff}{\mathrm{Aff}}

\DeclareMathOperator*\colim{colim}
\DeclareMathAlphabet{\mathrmsl}{OT1}{cmr}{m}{sl}

\newcommand{\rssymb}[2]{\newcommand{#1}{\mathrmsl{#2}} }
\newcommand{\oper}[3][n]{\newcommand{#2}{\mathop{\mathrm{#3}}%
\ifx n#1\nolimits\else\limits\fi} }
\newcommand{\rsoper}[3][n]{\newcommand{#2}{\mathop{\mathrmsl{#3}}%
\ifx n#1\nolimits\else\limits\fi} }

\oper\Ad{Ad}
\oper\ad{ad}

\oper\val{val}
\oper\coker{coker}
\oper\mult{mult}
\oper\Iso{Iso}
\oper\End{End}
\oper\Aut{Aut}
\oper\Sub{Sub}
\oper\Alt{Alt}
\oper\Ext{Ext}
\oper\Pic {Pic}
\oper\Sym{Sym}
\oper\Spec{Spec}
\oper\Spf{Spf}
\oper\Sp{Sp}
\oper\Spa{Spa}
\oper\Proj{Proj}
\rsoper\divg{div}
\rsoper{\sym}{sym}
\rsoper{\alt}{alt}
\rsoper\trace{tr}
\rssymb\id{id}

\newcommand{\thismonth}{\ifcase\month\or
  January\or February\or March\or April\or May\or June\or
  July\or August\or September\or October\or November\or December\fi
  \space\number\year}

\usepackage[xr,user]{zref}
\zexternaldocument*[I_]{poly}

\newcommand{\Poly}{\mathbf{Poly}}

\newcommand{\Alg}{\mathrm{Alg}}

\newcommand{\Pro}{\mathrm{Pro}}
\newcommand{\Sh}{\mathrm{Sh}}

\newcommand{\rigless}{\setminus\hspace{-5pt}\setminus}

\title{Affine manifolds are rigid analytic spaces in characteristic one \\ II: Analytic geometry and overconvergence}
\author{Andrew W. Macpherson}
\begin{document}

\maketitle
\begin{abstract}I extend the framework of rigid analytic geometry to the setting of algebraic geometry relative to monoids, and study the associated notions of separated, proper, and overconvergent morphisms.

The category of affine manifolds embeds as a subcategory defined by simple algebraic (normal) and topological (overconvergent) criteria. The affine manifold of a rigid space can be recovered either as a set of `Novikov field' points or as a universal Hausdorff quotient. After base change to any topological field, one obtains a `toric' analytic space that fibres over the affine manifold.
\end{abstract}

\tableofcontents

\section{Introduction}
The idea that degenerations of complex manifolds can be studied using torus fibrations over an affine manifold originates in the work of Hitchin \cite{Hitchin} and the SYZ conjecture \cite{SYZ} in mirror symmetry. In \cite{GrSi1}, the authors of the Gross-Siebert programme made this notion precise using logarithmic geometry. 

The purpose of this paper is to begin the development of the same ideas instead in the setting of \emph{non-Archimedean} geometry, as proposed in \cite{KoSo2}. Our approach will be, in continuation of the methods of part I \cite{part1}, to define a version of rigid analytic geometry purely in terms of multiplicative monoids. The punchline of the paper is that the resulting category actually \emph{includes} the category of affine manifolds; the monomials in the co-ordinate monoids are, symbolically, the exponentials of the affine functions.

In fact, the Raynaud-style approach to rigid analytic geometry taken here is sufficiently modular that the theory we obtain is strictly a generalisation of ordinary rigid analytic geometry over $\Z$: the latter can be recovered by simply plugging in the category of rings where, in this paper, we put the category of monoids with zero. The same is true for the theory of overconvergence introduced in \cite{part1} and continued in this paper. For a more detailed discussion of these features, see \S\ref{STRUCT} and \cite[\S\zref{I_STRUCT}]{part1}.

This modularity also allows us to obtain a family of well-behaved base change functors from the category of analytic spaces over a `valuation $\F_1$-field' to any topological field, parametrised by the open unit disc of that field.

\subsection*{Collages}
Let $\Delta\subseteq N$ be a rational, strongly convex polyhedron in a $\Z$-affine space. In \S\ref{POLY}, we make the elementary observation that the structures carried by the monoid $\Aff_\Delta(N,\Z)$ of affine functions on $N$ that are bounded above on $\Delta$ are essentially the data of a \emph{normal Banach algebra} $\F_1(\!(t)\!)\{\Delta\}$ of finite type over the discrete valuation $\F_1$-field $\F_1(\!(t)\!)$. 

This produces an invertible correspondence
\[ \{\text{polyhedra}\} \quad\leftrightarrow\quad \{\text{normal affine rigid analytic spaces of finite type over $\F_1(\!(t)\!)$}\}. \]
After base changing to a non-Archimedean field $K$ with uniformiser $t$, we obtain a rigid analytic space $\Spec K\{\Delta\}$ which, in usual terminology, is a \emph{rational domain} in affine space with polyhedron of convergence given by $\Delta$. It also carries a natural action by a unitary group $\mathrm{U}_1\tens N$.

Any normal analytic space, locally of finite type over $\F_1(\!(t)\!)$ is therefore obtained by glueing together the spectra of various algebras of the form $\F_1(\!(t)\!)\{\Delta\}$. The global combinatorial object must be obtained by glueing together overlapping embedded polyhedra. I have called these objects \emph{collages} in embedded polyhedra. The immediate conclusion is then (cf. \cite[cor. \zref{I_PFAN_EQUIVALENCE}]{part1}):
\begin{thrm}[\ref{AFF_COLL_THM}]The convergence complex contruction induces an equivalence 
\[ \Delta_{-/\Z}: \mathbf{Rig}_{\F_1(\!(t)\!)}^\mathrm{ltf/n/nb}\tilde\longrightarrow \mathrm{C}\mathbf{Poly}_\Z^N \]
between the category of normal rigid analytic spaces locally of finite type over $\F_1(\!(t)\!)$ and the category of collages in embedded lattice polyhedra.

A family of open subsets $U_i\subseteq X$ is a covering if and only if on every polyhedron $\Delta$ of $\Delta_{X/\Z}$ there is a finite refinement such that $\Delta(\Q)=\bigcup_i\Delta\cap\Delta_{U_i/\Z}(\Q)$.\end{thrm}

A $\Z$-affine manifold has a natural notion of embedded lattice polyhedron. Using this, one can easily realise affine manifolds as particularly nice collages, and therefore as rigid analytic spaces. Thus we reach the title result of the series:

\begin{thrm}[\ref{AFF_AFF}]The category of $\Z$-affine manifolds embeds as the full subcategory of the category of rigid analytic spaces over $\F_1(\!(t)\!)$ whose objects are boundaryless, overconvergent, and locally of finite type. Affine open subsets of the rigid space correspond to compact polyhedra inside the affine manifold.\end{thrm}

Much like for punctured cone complexes, there is also a natural notion of developing map $\delta:\Delta\rightarrow N$ for collages. The new feature here is that $N$, being an affine space, is in particular a collage, and the developing map is actually a local immersion of collages. If $\Delta$ is also an affine manifold, then this recovers the classical notion of development.

It is not too hard to combinatorially classify the points of the analytic spaces associated to polyhedra, and hence, by extension, collages. The classification involves a division into `types' generalising Berkovich's language for describing the points of non-Archimedean curves. The calculation can be found in the appendix \ref{APPEND}.

\subsection*{Overconvergence}
Passing from the world of real or complex analytic geometry to that of rigid analytic geometry, one quickly encounters an alarming profligation of new phenomena: even under local finiteness conditions, rigid analytic spaces can manifest all kinds of pathological topological features. 

Remarkably, a single extra stipulation - that of \emph{overconvergence} - when applied everywhere, simplifies matters to the point that for many purposes, we may pretend we are once again doing complex topology. For example, Deligne \cite{Deligne} introduced the notion of overconvergent open immersion to recover Berkovich's Hausdorff topology on a rigid analytic space. This definition was later generalised to arbitrary morphisms, under the name `partially proper', by Huber \cite[\S8]{Hubook}. Overconvergence, in a slightly different guise, is also essential to the study of $p$-adic cohomology of non-compact geometries.



On a slightly more down-to-earth level, in the first part \cite{part1} of this sequence we saw that the overconvergence condition on a formal scheme implies that the parametrising punctured cone complex is actually a \emph{manifold} without boundary. The same logic applies to collages: overconvergence over $\F_1(\!(t)\!)$ forces the real points of the collage to be an affine manifold, and this can be used to recover theorem \ref{AFF_AFF}. 

However, in the analytic regime we even have a slightly more direct route available: following Deligne, we may re-topologise $X$ using only the overconvergent open immersions to obtain, under fairly general circumstances, a Hausdorff topological space $X^\mathrm{sur}$ and universal separation map
\[ b:X\rightarrow X^\mathrm{sur}.\] Given its definition, it can hardly be surprising that the properties of $X^\mathrm{sur}$ are related to absolute overconvergence of $X$. 

The non-trivial - though by no means difficult to prove - observation is that the topological realisation of the collage associated to $X$ is then actually \emph{homeomorphic} with $X^\mathrm{sur}$. In light of this fact, the arguments leading up to theorem \ref{AFF_AFF} become rather tautological.
\begin{thrm}[\ref{AFF_OVERTHM}]Let $X$ be a normal rigid space, locally of finite type over $\F_1(\!(t)\!)$, with associated collage $\Delta_X$. There is a natural map $\Delta_X(\R_\infty)\rightarrow X$, and the composition
\[ \Delta_X(\R_\infty)\rightarrow X\rightarrow X^\mathrm{sur} \]
is a homeomorphism.\end{thrm}

In particular, if $X$ is overconvergent, then $X^\mathrm{sur}$ is actually a manifold.
In other words, starting from an algebraically-defined category and applying principles of pure rigid analytic geometry, we obtain a class of Hausdorff topological manifolds that has been studied by manifold topologists since time immemorial \cite{Goldman1,Goldman2}.

Well-understood principles that govern the latter may therefore shed light on the former (and, perhaps, vice versa). For instance, compact affine manifolds that are \emph{complete} - a condition conjecturally equivalent to a Calabi-Yau property - have been classified in dimensions up to three.

Finally, in \S\ref{C} we produce a base change from the overconvergent site of rigid analytic spaces over $\F_1(\!(t)\!)$ to the category of complex analytic spaces that recovers the classical construction
\[ \mu:TB/\Lambda^\vee\rightarrow B \]
of torus fibrations over affine manifolds. For the simple geometries that this statement concerns, it makes precise the idea that overconvergent topology `looks like' the topology of a complex analytic space.

\subsection*{B-model torus fibration}
Many ideas here have been motivated, if indirectly, by the mirror dual construction in symplectic geometry, features of which I expect will continue to inspire future work.

The starting point is the symplectic theory of toric manifolds, which revolves around the Delzant construction. This construction parametrises a symplectic manifold $(X,\omega)$ with a Hamiltonian action of a compact torus $T$ in terms of a completely integrable system
\[ X\rightarrow \Delta \]
with $\Delta$ a polyhedron inside an affine space modelled on the dual Lie algebra $\lie t^\vee$ of $T$. One can pass to a `large radius limit' in which $\Delta$ is a partial compactification of $\lie t^\vee$. If $X$ carries a complex structure, then this large polyhedron is controlled by the fan of $X$ - alternatively, it is another avatar for the $\F_1$-structure given by the torus embedding.

A more general situation that arises in SYZ mirror symmetry is when $X$ has the structure of a Lagrangian torus fibration over a manifold $B$. The $T$-action and embedding into an affine space under $\lie t^\vee$ are now only defined locally on $B$. Globally, the base attains a reduction of structure group to $\SL_n(\Z)\ltimes\R^n$, making it into an $\R$-\emph{affine manifold}. With a rational symplectic form and choice of pre-quantum structure on $X$, this can be refined to a $\Q$-affine structure by considering the `Bohr-Sommerfeld' Lagrangians.

The SYZ conjecture predicts that via a `Legendre dual' construction, $B$ can also be thought of as parametrising a certain maximal degeneration of Calabi-Yau varieties, and in particular, a rigid analytic space $X^\vee$ over $\C(\!(t)\!)$. The modern form of this conjecture has this dual `parametrisation' a kind of \emph{non-Archimedean torus fibration}
\[ \mu:X^\vee\rightarrow B \]
as in \cite{KoSo2} (where the concept is made precise using the Berkovich visualisation of $X^\vee$). This fibration carries a locally defined action by the dual torus with Lie algebra $\lie t^\vee$.

The thesis of this work is that non-Archimedean torus fibrations are the natural generalisation of toric geometry to the rigid analytic world. Though we focussed on $\F_1$ in this paper, I will return to the geometry of $X^\vee$ itself in a future work.

By further analogy with the symplectic story, one might also imagine generalisations to integrable systems with singularities; this would allow us to study a much larger class of analytic spaces. Indeed, certain singularities of `focus-focus' type are already allowed in the Gross-Siebert programme, which is concerned with fairly general maximally degenerate Calabi-Yau manifolds. However, a discussion of these ideas is far beyond the scope of the present work.

\subsection{On defining rigid analytic geometry}\label{STRUCT}
According to Raynaud \cite{Raynaud}, a rigid analytic space is what you get when you puncture a formal scheme along a closed subscheme. Following this principle, we define rigid analytic geometry by a certain localisation procedure applied to a category ${}_Z\mathbf{FSch}$ of formal schemes marked with a family of closed subschemes $Z$ that was discussed in \S I.\zref{I_FSCH_MARKING}.

The morphisms that get inverted are, intuitively, those birational maps that are isomorphisms `away from $Z$'. Such maps are generated by what are usually termed `admissible' modifications. Our procedure \S\ref{RIG_RAYNAUD} differs from simply inverting these morphisms (like in \cite{Raynaud}) only in that we force it to be compatible with \emph{glueing}; that is, with the structure of what will be the \emph{rigid topos} $\Sh\mathbf{Rig}$. Literally, all this means is that our localisation functor
\[ \Sh{}_Z\mathbf{FSch}\rightarrow \Sh\mathbf{Rig} \]preserves \emph{colimits}. This localisation procedure is almost identical to that of \cite{FujiKato}, though I have plumped for a more standard language to describe it. 

On the other hand, our input category ${}_Z\mathbf{FSch}$ is a generalisation of that considered in \emph{op. cit.}, the primary improvements it provides being that $\mathbf{Rig}$ contains the category of formal schemes, and that the map \[ j:X\rightarrow X^+ \] exhibiting $X^+$ as a formal model of $X$ is actually a morphism of rigid spaces. Analytic spaces in the traditional sense - that is, for which $Z$ is locally a reduction of a formal model - are called \emph{purely} analytic.

One can reproduce many of the basic definitions and arguments of rigid geometry by means of a `lift' to ${}_Z\mathbf{FSch}$ (\S\ref{RIG_MODELS}).

\

The Raynaud-style definition has the advantage of clearly producing the category of objects we want to study, but it lacks a certain concreteness. It will also be useful to have local descriptions of rigid analytic spaces in terms of \emph{topological commutative algebra}; in particular, this will be essential to get any kind of module theory (though we don't pursue that in this paper). In this paper, I only discuss topological $\F_1$-algebras.

This local algebra is where the $\F_1$-regime enjoys considerable simplifications relative to the $\Z$-regime, the main point of departure being that as soon as a module is \emph{Hausdorff}, it is \emph{complete}. We are therefore able to switch between a point-set-topological and pro-object description of topological $\F_1$-algebras and modules.\footnote{In general, the former description breaks down over $\Z$, leading authors \cite{Abbes,FujiKato} to impose fundamental finiteness conditions at an early stage.} In \S\ref{RIG_LC}, I review the various elementary assumptions one needs to get the theory off the ground, along with some remarks their geometric provenance. One arrives at a certain category of `Tate' algebras, that is, pairs $(A;A^+)$ consisting of a pro-discrete ring $A^+$ and a localisation $A$ thereof. 

The spectrum of $A^+$ is a formal scheme, and one marks the divisors cut out by the functions inverted by the localisation $A^+\rightarrow A$. The rigid space obtained by puncturing this marking is labelled $\Spec A$ (\S\ref{RIG_AFFINE}). This produces a functor
\[\Spec:\{\text{topological rings}\}\rightarrow\mathbf{Rig}.\]
The only thing remaining to recover the picture of `affine objects' familiar from algebraic geometry is to invert the construction. It turns out that this is not quite possible: there are non-trivial admissible modifications between affine marked formal schemes.

We can salvage the situation by identifying exactly which morphisms are inverted at the level of algebra. The algebraic detail you have to know to get this to work is:
\begin{itemize}\item[] What happens to $\Gamma\sh O_{X^+}$ under an admissible blow-up (of a model $X^+$ of $\Spec A$)?\end{itemize}
By finiteness of global sections over projective morphisms, the answer is that you get an \emph{integral algebra extension} inside $A$ (\S I.\zref{I_RIG_LEMMA}). Thus the full, localising subcategory of pairs $(A;A^+)$ such that $A^+$ is \emph{integrally closed} inside $A$ embeds fully faithfully via Spec into $\mathbf{Rig}$.


\

General nonsense (cf. \S I.\zref{I_TOPOS}) also provides us with an underlying topological space for our rigid analytic spaces. This is the famous \emph{Riemann-Zariski space}, defined here \S\ref{RIG_RZ} as a limit over all formal models (with a more explicit description, in a special case, in the appendix \ref{APPEND}). If one didn't already know to look at this space, one could still find its definition by inspecting the nuts and bolts of the localisation construction.


\section{Remarks on polyhedra}\label{POLY}

Let $H\subseteq\R$ be a totally ordered additive group, and write $H^\circ:=H\cap\R_{\leq0}$. An \emph{$H$-rational polyhedron} is a subset of an $H$-affine space $N$ defined by a finite list of inequalities with coefficients in $H$.

We will be interested in the filtered monoids that arise as sets of affine functions bounded above on a polyhedron, and how they can be used to recover the combinatorial structure of the same polyhedra.

Let $N$ be an affine space over $H$, and let $\Delta\subseteq N$ be an $H$-rational polyhedron. We introduce the partially ordered monoids
\[ \Aff^+_\Delta(N,H) \quad\subseteq\quad    \Aff_\Delta(N,H) \quad\subseteq\quad \Aff(N,H) \]
of, in reverse order, affine functions on $N$ with integral slopes, affine functions \emph{bounded above} on $\Delta$, and affine functions bounded above by zero on $\Delta$. For the moment, I will be deliberately vague about what kind of objects $N$ and $\Delta$ are, for the most part considering them as formally dual to the partially ordered monoids in which they are encoded.

\subsection{The ambient affine space}

The monoid $\Aff(N,H)$ is actually a torsion-free Abelian group, which fits into an exact sequence
\[0 \rightarrow H\rightarrow \Aff(N,H) \rightarrow \Lambda_{N/H} \rightarrow 0 \]
with $\Lambda_{N/H}$ a lattice: the \emph{character lattice} of $N$. The image in $\Lambda_{N/H}$ of a function $F\in\Aff(N,H)$ is its \emph{differential} $dF$.

We can recover the $H$-rational points of $N$ from its affine functions by the formula
\[ N(H)=\Hom_H(\Aff(N,H),H), \]
where $\Hom_H$ denotes the set of group homomorphisms that commute with the structural maps from $H$. This set is a torsor for $\Hom(\Lambda_{N/H},H)$; in other words, $\Lambda_{N/H}^\vee\tens H$ is the model space for $N(H)$. More generally, we may take points in any additive extension $H\subseteq H^\prime\subseteq \R$ of $H$ with the formula
\[ N(H^\prime)=\Hom_H(\Aff(N,H),H^\prime). \]

If $N_1\rightarrow N_2$ is an affine map of $H$-affine spaces, then we can form the exact sequence
\[ 0\rightarrow \nu^\vee_{N_1/N_2} \rightarrow \Aff(N_2,H) \rightarrow \Aff(N_1,H) \rightarrow \Lambda_{N_1/N_2} \rightarrow 0 \]
with $\Lambda_{N_1/N_2}^\vee\tens H$ acting simply transitively on the fibres.
If $\nu^\vee_{N_1/N_2}=\Lambda_{N_1/N_2}^\vee=0$, then $N_1\rightarrow N_2$ is a \emph{lattice refinement} of affine spaces. An extension $H^\prime$ of $H$ induces a natural refinement \[N\rightarrow N\tens H^\prime ,\quad \Aff(N\tens H^\prime,H^\prime)=\Aff(N,H)\oplus_HH^\prime \] of any $H$-affine space $N$ (provided $H$ is non-trivial). Of course, $N\tens H^\prime(H^\prime)=N(H^\prime)$.

We can form the quotient $N_2/\Lambda^\vee$ by a primitive distribution $\Lambda^\vee\subseteq\Lambda_{N_2/H}^\vee$ with affine functions the $\Lambda^\vee\tens H$-invariants. They fit into a Cartesian square
\[ \xymatrix{ 
\Aff(N_2/\Lambda^\vee,H) \ar[r]\ar[d] & \ker[\Lambda_{N_2/H}\rightarrow\Lambda]\ar[d] \\ 
\Aff(N_2,H) \ar[r] & \Lambda_{N_2/H}
}\]
If the distribution $\Lambda=\Lambda_{N_1/H}$ comes from an affine subspace $N_2\subseteq N_1$, we get a further Cartesian square
\[ \xymatrix{ 
\Aff(N_2/N_1) \ar[r]\ar[d] & \Aff(N_2,H)\ar[d] \\ 
H \ar[r] & \Aff(N_1,H)
}\]
whose left-hand vertical arrow splits the usual cotangent sequence. Thus $\Aff(N_2/N_1)\cong H \oplus \nu^\vee_{N_1/N_2}$ and $N_2/N_1(H)$ is in bijection with its model $\nu_{N_1/N_2}\tens H$.

Any map of affine spaces can be factored into a surjection (purely transcendental submersion), a refinement (\'etale map), and a primitive embedding.

\subsection{The cone of bounded functions}

Let $\Delta\subseteq N$ be a polyhedron. The set $\Aff_\Delta(N,H)$ is a submonoid of $\Aff(N,H)$, with equality if $\Delta$ is of finite extent. When $\Delta$ is $H$-rational, $\Aff_\Delta(N,H)$ is a \emph{cone over $H$} - a saturated monoid generated by $H$ and finitely many additional elements. It fits into an exact sequence
\[ 0\rightarrow H\rightarrow \Aff_\Delta(N,H)\rightarrow\Lambda_{\Delta/H}\rightarrow 0\]
whose right-hand term $\Lambda_{\Delta/H}$ is a polyhedral cone (finitely generated, saturated subgroup) inside $\Lambda_{N/H}$. Its polar $\Lambda^\diamond_{\Delta/H}\subseteq\Lambda_{N/H}^\vee$ is usually called the \emph{recession cone} of $\Delta$. It is the set of tangent vectors to rational rays contained within $\Delta$.

The cone $\Aff_\Delta(N,H)$ is capable of separating points only up to the action of the \emph{lineality space}
\[ \Delta^|:=\left(\frac{\Aff(N,H)}{\Aff_\Delta(N,H)} \right)^\vee \subseteq \Lambda_{N/H}^\vee, \]
which is the largest linear subspace of $\Lambda_{\Delta/H}^\diamond$. Alternatively,
\[ \Hom_H(\Aff_\Delta(N,H)\tens\Z,H)= N/\Delta^|(H). \]
The lineality space is zero if and only if $\Aff_\Delta(N,H)\tens\Z=\Aff(N,H)$, if and only if the recession cone (equivalently $\Delta$) is \emph{strongly convex}.

If $N_1\rightarrow N_2$ maps $\Delta_1\subseteq N_1$ into $\Delta_2\subseteq N_2$, we once again have an exact sequence
\[ 0\rightarrow \nu^\diamond_{\Delta_1/\Delta_2} \rightarrow \Aff_{\Delta_2}(N_2,H) \rightarrow \Aff_{\Delta_1}(N_1,H)\rightarrow \Lambda_{\Delta_1/\Delta_2} \rightarrow 0 \] of saturated monoids whose outer terms $\nu^\diamond_{\Delta_1/\Delta_2},\Lambda_{\Delta_1/\Delta_2}$ are cones.

We may also form a quotient $\Delta_2/\Lambda^\vee$ of $\Delta_2$ by a distribution $\Lambda^\vee\subseteq\Lambda^\vee_{N_2/H}$, whose $H$-points form the image of $\Delta_2(H)$ in $N_2/\Lambda^\vee(H)$, via the fibre square
\[ \xymatrix{ 
\Aff_{\Delta_2/\Lambda^\vee}(N_2/\Lambda^\vee,H) \ar[r]\ar[d] &\Aff(N_2/\Lambda^\vee,H)\ar[d] \\ 
\Aff_{\Delta_2}(N_2,H) \ar[r] & \Aff(N_2,H)
}\]
If the distribution comes from an embedding $N_1\subseteq N_2$ of affine spaces, the affine functions on the quotient are $H\oplus\nu^\diamond_{\Delta_2/N_1}$.

Taking the quotient by the lineality space allows us to replace any polyhedron $\Delta$ with a strongly convex one $\Delta/\Delta^|$.
All of the polyhedra in this document will be strongly convex.

\subsection{The cone of non-positive functions}\label{POLY_NONPOS}

If $\Aff_\Delta(N,H)$ is generated by $H$ and affine functions $F_1,\ldots,F_k\in\Aff^+_\Delta(N,H)$, then $\Delta$ is an intersection of half-spaces
\[ \Delta(H)=\bigcap_{i=1}^kF^{-1}H^\circ. \] 
The recession cone $\Lambda^\diamond_{\Delta/H}$ is the largest submonoid of $\Lambda^\vee_{N/H}$ such that the action of $\Lambda^\diamond_{\Delta/H}\tens H^\circ$ on $N(H)$ preserves $\Delta(H)$. 
It follows that
\[ \Delta(H)=\Hom_{H^\circ}(\Aff^+_\Delta(N,H),H^\circ).\]
To put it another way, $\Delta(H)$ is the set of homomorphisms of \emph{pairs} \[\left(\Aff_\Delta(N,H);\Aff^+_\Delta(N,H)\right)\rightarrow (H;H^\circ)\]  that respect the $H$-structure.

The monoid $\Aff_\Delta(N,H)$ comes equipped with an $H$-indexed filtration
\[ \Aff^+_\Delta(N,H)+\lambda\hookrightarrow \Aff_\Delta(N,H), \quad \lambda\in H \]
by $\Aff^+_\Delta(N,H)$-invariant subsets. It is automatically preserved by homomorphisms of pairs over $H$. Since every element of $\Aff^+_\Delta(N,H)$ is be bounded above by some constant, the filtration is \emph{exhaustive}; since no function is bounded above by \emph{every} constant, it is also \emph{separated}.

\subsection{Morphisms over $\F_1$ and the boundary at infinity}\label{POLY_BOUNDARY}

The set $\Delta(H)$ sits inside a natural `compactification' $\Delta(H_\infty)$ in which certain strata, indexed by the faces of the recession cone, are added at infinity. Tthese strata are defined by allowing bounded functions to take the value $-\infty$. To describe this algebraically, we need to introduce \emph{absorbing elements} to our monoids - in other words, move over $\F_1$.

Let $Q$ be a monoid, $\F_1[z^Q]$ the associated $\F_1$-algebra. Let $I\trianglelefteq\F_1[z^Q]$ be an ideal. The quotient of $\F_1[z^Q]$ by $I$ is, as a set, obtained by identifying the elements of $I\trianglelefteq\F_1[z^Q]$ with $0$. In particular, \[\F_1[z^Q]\setminus I\rightarrow(\F_1[z^Q]/I)\setminus 0\] is \emph{bijective}.
If in particular $I=\lie p$ is prime, then $\sigma_\lie{p}:=\F_1[z^Q]\setminus\lie p$ is a submonoid of $Q$; in fact, it is a \emph{face} in the sense that \[X+Y\in\sigma_\lie{p} \quad \Leftrightarrow \quad X,Y\in\sigma_\lie{p}. \]
In particular, $\sigma_\lie{p}$ contains every subgroup of $Q$. The quotient is uniquely identified with
\[ \F_1[z^Q]\rightarrow \F_1[z^{\sigma_{\lie p}}],\quad z^X \mapsto \left\{\begin{matrix} z^X \quad\text{if }X\in\sigma_\lie{p} \\ 0\quad\text{otherwise.} \end{matrix}\right. \]

Specialise now to the case of the $\F_1$-algebra $\sh O\{\Delta\}$ associated to $\Aff_\Delta(N,H)$. Since $H\subset\sh O\{\Delta\}$ is a group, it is contained in the complement of $\lie p$. This complement is therefore the preimage of its image in the corecession cone $\Lambda_{\Delta/H}$, of which it is a face $\mathrm{asy}^\diamond_\lie{p}$. The quotient is identified with $\sh O\{\Delta/\mathrm{asy}_{\lie p}\}$, where $\mathrm{asy}_\lie{p}$ is the polar cone to $\mathrm{asy}^\diamond_\lie{p}$. 

The cone $\mathrm{asy}_\lie{p}$ could be called the \emph{asymptotic cone} of $\Delta$ with limits in $\Delta/\mathrm{asy}_{\lie p}$.

The sense in which $\Delta(H_\infty)$ is a `compactification' of $\Delta(H)$ is as follows: let $\{x_n\}_{n\in\N}$ be a sequence that escapes every bounded subset. Then if $\mathrm{asy}_{\lie p}$ is the minimal asymptotic cone to $x_n$, the sequence can be said to \emph{limit} into the infinite face $\Delta/\mathrm{asy}_{\lie p}$ with affine functions $\sh O(\Delta)/\lie p$; to be precise, its limit is the (eventually constant) image of the tail end of the sequence in $\Delta/\mathrm{asy}_{\lie p}(H)$.
This argument also shows that $\Delta(\R_\infty)$, equipped with the order topology coming from $\R$, is compact in the usual sense.

\subsection{Faces}\label{POLY_FACE}

Let $v\in\Lambda_{\Delta/H}$. We can associate to $v$ a \emph{finite face} $\Delta_v$ of $\Delta$ along which any affine function integrating $v$ is maximised. The function can be normalised to vanish along $\Delta_v$; it is then in particular a member of $\Aff^+_\Delta(N,H)$. There is a pushout square
\[\xymatrix{
\nu^\diamond_{\Delta_v/\Delta} \ar[r]\ar[d] & \Aff^+_\Delta(N,H) \ar[d] \\
\nu^\diamond_{\Delta_v/\Delta}\tens\Z \ar[r] & \Aff^+_{\Delta_v}(N,H)
}\]
with $\nu^\diamond_{\Delta_v/\Delta}$ a cone in $\Aff^+_\Delta(N,H)$.

Looking at it another way, the finite faces of $\Delta$ are exactly the sub-polyhedra that are cut out by localisations of the $\F_1$-algebra $\sh O^+\{\Delta\}$ associated to $\Aff^+_\Delta(N,H)$. The character cone $\Lambda_{\Delta/H}$ is naturally subdivided by the corresponding conormal cones. By taking the cone over $\Delta$, this recovers the familiar classification of open subsets of toric varieties discussed in I.\zref{I_FAN_CONE}. 

The usual convention is that faces in dimensions zero and one are called \emph{vertices} and \emph{edges}, and those in codimensions one and two are called \emph{facets} and \emph{ridges}, respectively.

\paragraph{Variant - Faces as closed strata at $t=0$}
Since $\nu^\diamond_{\Delta_v/\Delta}$ is a face of $\Aff^+_\Delta(N,H)$, its complement is a prime ideal in the corresponding $\F_1$-algebra $\sh O^+\{\Delta\}$. The quotient is dual to a closed subvariety of finite type over the residue field $\F_1$. The open subset defined by the localisation above is the minimal open subset of $\Delta$ that contains this closed subvariety.

\subsection{Affine manifolds}\label{POLY_AFF}

An \emph{$H$-affine manifold} is a manifold equipped with a reduction of structure group to the $H$-affine group $\GL_n(\Z)\ltimes H^n$. In other words, it is a manifold locally modelled on subsets of an $H$-affine space with affine transition maps.

Such a manifold comes equipped with a locally constant sheaf $\Aff(B,H)$ of \emph{$H$-affine functions}. If $H^\prime$ is an extension of $H$, it is possible to make sense of the \emph{$H^\prime$-points} $B(H^\prime)$ of $B$ as the set of points on which $H$-affine functions take values in $H^\prime$. In particular, $B(\R)$ is just the underlying topological space of $B$, equipped with the weak topology.

Let $\widetilde B$ be a universal cover of $B$. The global sections of $\Aff(\widetilde B,H)$ are dual to an $n$-dimensional affine space $N$, and local charts patch together to yield a natural local affine diffeomorphism
\[ \delta:\widetilde B\rightarrow N, \]
called the \emph{developing map} of $B$. The obstruction to this map descending to $B$ is the \emph{monodromy representation}
\[ \rho:\pi_1(B,p)\rightarrow\Aut(\Aff(N,H)) \]
for any $p\in B(\R)$. So this representation is trivial if and only if $B$ is a covering space of an open subset of $N$. Of course, we can always make sense of the developing map locally on $B$.

Suppose that $B$ is connected. Then the developing map is surjective if and only if all geodesics on $B$ are parametrised by the entire real line, i.e. if $B$ is \emph{complete}. In this case, $\widetilde B\cong N$ and $B$ is a $K(\pi,1)$.

The significance of the completeness hypothesis to analytic geometry is the conjecture of Markus:
\begin{conj}[Markus]An affine manifold is complete if and only if it has parallel volume.\end{conj}
Certainly every complete orientable affine manifold has parallel volume; the converse is still open. Affine manifolds with parallel volume correspond to \emph{Calabi-Yau} analytic spaces. 


Closed, complete affine manifolds are classified by groups acting cocompactly and properly discontinuously by affine transformations on Euclidean space. Such subgroups of $\Aff(N)$ are also known as \emph{affine crystallographic groups}. A classification is known in dimensions up to three; see \cite{Goldman2}.

It is also not difficult to formulate a notion of affine manifold \emph{with corners at infinity}; I provide a sketch-definition here. An affine manifold with corners is a manifold with corners $(B,\partial B)$ with an affine structure on $B\setminus\partial B$ such that
\begin{itemize}
\item a geodesic limits into the boundary only when it is complete;
\item each boundary stratum \[\partial B_k \stackrel{i}{\hookrightarrow} B \stackrel{j}{\hookleftarrow} B\setminus\partial B_k\] is an $H$-affine manifold with respect to the subsheaf $\mathrm{Aff}(\partial B_k,H)\subseteq i^*j_*\mathrm{Aff}(B,H)$ of locally bounded sections.
\end{itemize}
Affine functions on an affine manifold with corners are allowed to take the value $-\infty$. If the function is not everywhere $-\infty$, it can only take this value on the boundary.

\section{Rigid analytic spaces}\label{RIG}
This presentation of analytic spaces is valid \emph{verbatim} for ordinary rigid analytic spaces over $\Z$. For that reason, I have suppressed the subscript $\F_1$s for this and the next two sections.

The main examples having already been exposed in section \ref{POLY}, I have kept the discussion here mostly theoretical. I invite the reader also to keep in mind his favourite rigid analytic spaces over non-Archimedean fields as geometric motivation for this development.

\subsection{Raynaud presentation}\label{RIG_RAYNAUD}

In \S I.\zref{I_FSCH_MARKING} we defined a category ${}_Z\mathbf{FSch}$ of \emph{marked formal schemes}, whose objects are pairs $(X^+;Z)$ consisting of a formal scheme $X^+$ and a finitely presented closed formal subscheme $Z\subseteq X^+$; morphisms of pairs are morphisms of formal schemes that, up to nilpotents, pull back target markings into source markings. The isomorphism class of a pair depends only on $X^+$ and the underlying reduced formal scheme of $Z$.

We will construct the rigid topos $\Sh\mathbf{Rig}$ so that it is universal with respect to the existence of a quasi-compact, quasi-separated geometric morphism
\[ \Sh\mathbf{Rig} \stackrel{\zeta}{\longrightarrow} \Sh{}_Z\mathbf{FSch} \]
such that $\zeta^*=:(-)\rigless Z$ inverts isomorphisms `away from $Z$'. We will work on the principle that any such morphism can be dominated by a blow-up along a subscheme of $Z$, or more generally, a subscheme of $X^+$ whose reduction is contained in $Z$. Note that the definition of such morphisms is not local in $\Sh\mathbf{Sch}$, and so it will not be possible to consider formal schemes as embedded in that category.

The fact that $\zeta^*$ must be the pullback of a geometric morphism means that it is a localisation of $\Sh{}_Z\mathbf{FSch}$ only as a \emph{stack} on itself, rather than as a plain (external) category. That is to say, it must preserve glueings, a.k.a. colimits. Moreover, the hypothesis that $\zeta$ be qcqs - by definition, meaning that $\zeta^*$ preserves compact objects - is essential if we are to have a reasonable theory of compactness for rigid analytic spaces. 

Let us first define the category on compact objects - more precisely, on the site ${}_Z\mathbf{FSch}^\mathrm{qcqs}$ of qcqs marked formal schemes. Let us define $W$ to be the class of admissible modifications (def. I.\zref{I_RIG_ADMIT_DEF}). It is stable for composition. By lemma I.\zref{I_MORP_ADMISSIBLE}, the saturation of $W$ is stable for base change and descent.


We may now construct the category of \emph{qcqs rigid spaces} as a localisation
\[ {}_Z\mathbf{FSch}^\mathrm{qcqs}\rightarrow{}_Z\mathbf{FSch}^\mathrm{qcqs}[W^{-1}]=:\mathbf{Rig}^\mathrm{qcqs}. \]
The following construction works for any left exact localisation. Let $\Pro_Z\mathbf{FSch}^\mathrm{qcqs}$ denote the category of (`admissible') pro-objects of $\mathbf{FSch}^\mathrm{qcqs}$ with transition maps in $W$. Since the saturation of $W_\mathrm{loc}$ is stable for base change, this category is closed under finite limits in the category of all pro-objects. The inclusion of ${}_Z\mathbf{FSch}^\mathrm{qcqs}$ as the constant pro-objects is therefore left exact.

We will let $\mathbf{Rig}^\mathrm{qcqs}$ be the full subcategory of $\Pro_Z\mathbf{FSch}^\mathrm{qcqs}$ spanned by the \emph{colocal} objects, that is, the admissible pro-objects $F$ for which $\Hom(F,X)\rightarrow \Hom(F,Y)$ is a bijection for any $X\rightarrow Y$ in $W$. 

The inclusion of $\mathbf{Rig}^\mathrm{qcqs}$ has a right adjoint, which sends an admissible pro-object $\lim_iX_i$ to the cofiltered limit of all admissible blow-ups of the $X_i$.
\[\xymatrix{ & \mathbf{Rig}^\mathrm{qcqs}\ar@/_/[d] \\ {}_Z\mathbf{FSch}^\mathrm{qcqs} \ar[r] & \Pro_Z\mathbf{FSch}^\mathrm{qcqs}\ar@/_/[u] }\]
The composite functor ${}_Z\mathbf{FSch}^\mathrm{qcqs} \rightarrow \Pro_Z\mathbf{FSch}^\mathrm{qcqs}$, which sends a marked formal scheme $(X^+,Z)$ to the limit of blow-ups with centres in $Z$, is left exact.

\begin{lemma}Let $F:{}_Z\mathbf{FSch}^\mathrm{qcqs}\rightarrow\mathbf C$ be a functor inverting admissible blow-ups. It factors uniquely through $\mathbf{Rig}^\mathrm{qcqs}$.\end{lemma}
\begin{proof}There is a unique extension $\Pro_ZF$ of $F$ to $\Pro_Z\mathbf{FSch}^\mathrm{qcqs}$ that preserves limits of admissible blow-ups. This extension factors uniquely through the coreflective subcategory $\mathbf{Rig}^\mathrm{qcqs}$. Conversely, any factorisation of $F$ through $\mathbf{Rig}^\mathrm{qcqs}$ restricts along the coreflector to a functor $\Pro_Z\mathbf{FSch}^\mathrm{qcqs}\rightarrow\mathbf C$, which sends limits of admissible blow-ups to constant diagrams and hence is uniquely identified with $\Pro_ZF$.\end{proof}

\begin{lemma}[Local localisation]\label{RIG_TOPOS}Let $\mathbf S$ be a topos, $W$ a composable class of qcqs morphisms in $\mathbf S$. Let $\mathbf S^c$ be a finitely complete site for $\mathbf S$ whose objects are compact in $\mathbf S$ (so that $\mathbf S$ is coherent). Let $\zeta^*:\mathbf S^c\rightarrow\mathbf S^c[W^{-1}]$ be a left exact localisation.\footnote{To be precise, $\zeta^*$ is universal among functors to any category that invert the morphisms of $W$, and it is left exact. The existence of such a localisation depends only on the shape of $W$.}

There exists a coherent topos $\mathbf S_W$ and a qcqs geometric morphism $\zeta:\mathbf S_W\rightarrow\mathbf S$ whose pullback inverts morphisms represented by $W$, and is universal with respect to these properties. Moreover, there is a unique (up to unique isomorphism) commuting diagram
\[\xymatrix{\mathbf S^c \ar[r]^-{\zeta^*}\ar[d] & \mathbf S^c[W^{-1}] \ar[d]\\
\mathbf S \ar[r]^{\zeta^*} & \mathbf S_W}\]
that makes $\mathbf S^c[W^{-1}]$ into a finitely complete site of compact objects of $\mathbf S_W$.

If $\mathbf S^c$ is subcanonical and the morphisms in $W$ are local, then $\mathbf S^c[W^{-1}]$ is subcanonical.\end{lemma}
\begin{proof}This is probably well-known, but I include a proof here to save myself sifting through the literature. Put $\mathbf S_W$ the category of sheaves on $\mathbf S^c[W^{-1}]$ with respect to the coverage induced by $\zeta^*$. It is initial among topoi fitting into a commuting square as above. 

That the objects of $\mathbf S^c$ are compact in $\mathbf S$ is to say that the local isomorphisms  in $\mathrm{PSh}\mathbf S^c$ are qcqs, that is, the base change of any such to a compact presheaf is compact. Since the morphisms of $W$ are qcqs, so are the induced system of local isomorphisms in $\mathrm{PSh}(\mathbf S^c[W^{-1}])$. Hence  the objects of $\mathbf S^c[W^{-1}]$ are compact in $\mathbf S_W$, which is therefore coherent. Moreover, since the compact objects of $\mathbf S$ (resp. $\mathbf S_W$) are the finite colimits of representable objects, and $\zeta^*$ preserves these, the latter is qcqs.\end{proof}

Applying lemma \ref{RIG_TOPOS}, we obtain the \emph{rigid topos} $\Sh\mathbf{Rig}$. The definition of (represented) \emph{open immersion} in $\Sh{}_Z\mathbf{FSch}$ induces a corresponding notion in $\Sh\mathbf{Rig}$. For a more explicit description of this notion, see the end of \S\ref{RIG_MODELS}.

\begin{defn}[Rigid spaces as sheaves]A \emph{rigid analytic space} is a locally representable sheaf on $\mathbf{Rig}^\mathrm{qcqs}$. The category of rigid analytic spaces is denoted $\mathbf{Rig}$. \end{defn}

The above arguments show that a rigid analytic space can locally be understood as a \emph{pro-formal scheme}; for further remarks on this perspective, see \S\ref{RIG_RZ}.


In traditional rigid analytic geometry, one usually restricts attention to formal schemes which are punctured along the entirety of their reduction.

\begin{defn}[Purely analytic]\label{RIG_PURELY}We call a rigid analytic space \emph{purely analytic} if it admits no morphism from any formal scheme, that is, if $\zeta_*X=\emptyset$. A rigid space is purely analytic if and only if it locally has a formal model on which all algebraic subschemes are marked.\end{defn}

\paragraph{Base change} Recall that in \S I.\zref{I_RIG_FSCH} we have defined a spatial geometric morphism
\[ \pi:\Sh\mathbf{FSch}_\Z\rightarrow\Sh\mathbf{FSch}_{\F_1} \]
so that $\F_1$-formal schemes may be base changed to ordinary formal schemes over $\Z$. By part \emph{i)} of lemma I.\zref{I_MORP_ADMISSIBLE}, up to saturation, this pullback functor also preserves $W$. It therefore descends to a spatial geometric morphism
\[\pi:\Sh\mathbf{Rig}_\Z\rightarrow\Sh\mathbf{Rig}_{\F_1},\qquad \pi^*:\mathbf{Rig}_{\F_1}\rightarrow\mathbf{Rig}_{\Z} \]
so that rigid analytic spaces may be base changed from $\F_1$ to $\Z$.

A case of particular interest is the following. Let $\F_1(\!(t)\!)$ be the Laurent series field with the $t$-adic topology, $K$ a non-Archimedean field (in the usual sense) equipped with a topological nilpotent $t$. There is a unique map $\F_1(\!(t)\!)\rightarrow K$ that preserves $t$. This gives a family of base change operations
\[\mathbf{Rig}_{\F_1(\!(t)\!)}\rightarrow\mathbf{Rig}_{K} \]
parametrised by the open unit disc of $K$.


\subsection{Models}\label{RIG_MODELS}

The local existence of formal schemes modelling rigid spaces is fundamental to all aspects of their study.

A \emph{model} for an object $X$ of $\Sh\mathbf{Rig}$ is an object of $\Sh_Z\mathbf{FSch}$ together with an isomorphism $X^+\setminus Z\cong X$. If $X\in\mathbf{Rig}$, we will want to assume that models of $X$ are formal schemes. Models of $X$ form a category $\mathbf{Mdl}_X$, the fibre of $\zeta^*$ over $X$. We say that $X$ has \emph{enough models} if $\mathbf{Mdl}_X$ is cofiltered.

If $X$ is qcqs, it has enough models. More generally:

\begin{prop}A paracompact rigid space has enough models.\end{prop}

The natural (counit) map $(X^+,Z)\rightarrow(X^+,\emptyset)$ in $\Sh_Z\mathbf{FSch}$ pulls back to a morphism \[j:X\rightarrow X^+\] in $\Sh\mathbf{Rig}$; conversely, this morphism characterises the model $(X^+,Z)$, as we set $Z$ equal to the union of all closed subschemes of $X^+$ whose pullback to $X^+$ is empty. We will usually denote the data of the formal model in terms of this morphism $j$.



\begin{defns}[Representability conditions]\label{RIG_REPS}A morphism $f:X\rightarrow Y$ in $\Sh\mathbf{Rig}$ is said to be \emph{representable by formal schemes} if, for any qcqs rigid space $Z$ and morphism $Z\rightarrow Y$, there exists a model of the base change $X\times_YZ\rightarrow Z$ such that the square
\[\xymatrix{
X\times_YZ \ar[r]\ar[d] & (X\times_YZ)^+\ar[d]^{f^+} \\
Z \ar[r]    & Z^+
}\]
is Cartesian. In particular, $f$ admits models locally on $Y$. As a morphism in $\Sh_Z\mathbf{FSch}$, $f^+$ is represented by a formal scheme in the sense of \S\ref{RIG_RAYNAUD}. If, more generally, $X$ is a union of open subobjects that are representable by formal schemes over $Y$, we say that $f$ is \emph{locally representable by formal schemes}.

Let $\mathbf P$ be a property (resp. source-local property) of morphisms in $\Sh\mathbf{FSch}$. We say that a morphism in $\Sh\mathbf{Rig}$ \emph{has property ${}^+\mathbf P$} if it is representable (resp. locally representable) by formal schemes, and moreover the (local) models $f^+$ can be chosen having property $\mathbf P$.

With $\mathbf P$ equal to one of the following properties, we suppress the superscript plus:
\begin{enumerate}\item representable by schemes;
\item open immersion;
\item locally of finite type, presentation
\item (formal) embedding;
\item (formally) finite, integral
\item (formally) projective.\end{enumerate}So, for example, a morphism of rigid spaces is projective if it admits a projective model. We specifically discuss the cases of separated, overconvergent, and proper morphisms in \S\ref{SEP}.
\end{defns}

Suppose that $X$ admits enough models. Any quasi-compact open immersion $U\hookrightarrow X$ admits a global model, and is thus given by the data of a model $X^+$ of $X$ and a quasi-compact open subset $U^+$ of $X^+$. Note that even if $X$ is ${}^+$Noetherian, so that $X^+$ is a Noetherian topological space, $X$ itself typically has many non-quasi-compact open subsets.

\subsection{Topologies on $\F_1$-algebras}\label{RIG_LC}
The remarks of this section and the next are valid in principle for $\F_1$ and $\Z$, but to avoid technicalities I am only claiming precision for the $\F_1$ case. The study of the $\Z$ case is in any case well-established \cite{Abbes,FujiKato}.

As in \S I.\zref{I_MORP_NORM}, we will consider pairs $(A;A^+)$ consisting of an $\F_1$-algebra $A^+$ and an $A^+$-algebra $A$. Subsets of $A$ invariant under the \emph{ring of integers} $A^+$ will be called \emph{discs}.

Except where otherwise specified, our pairs will satisfy the following assumption:
\begin{itemize}\item $A^+$ is a saturated submonoid of $A$. \hfill \emph{relatively normal} \end{itemize}
In other words, if $f\in A$ and $f^n\in A^+$, then $f\in A^+$. By lemma I.\zref{I_RIG_THELEMMA}, any affine $A$-admissible modification of $\Spec A^+$ sits between $A^+$ and its saturation in $A$.

\

A \emph{locally convex topology} on $(A;A^+)$ is a separated filtration by discs, which are declared \emph{open}, which as in \S I.\zref{I_RIG_PROD} we close under intersection and enlargement of discs. In particular, $A^+$ attains under the induced topology the structure of a linearly topologised ring. In fact, \emph{in light of lemma I.\zref{I_RIG_PRO}}, $A$ is a pro-discrete $A^+$-module.\footnote{Emphasis added here to highlight the fact that this is \emph{not} true over $\Z$.}

As before, we will always want to assume
\begin{itemize}
\item the product of two open discs is open. \hfill \emph{weak adicity}\end{itemize}
This implies that $A^+$ is an adic ring in the sense of \S I.\zref{I_RIG_PROD}. 

\begin{remark}\label{RIG_ADIC}The reason I have added the prefix \emph{weak} is that we might strengthen it to the condition
\begin{itemize}\item $A^+$ - the empty product of discs - is open in $A$. \hfill \emph{adicity}\end{itemize}
which is traditional, and perhaps aesthetically pleasing as a complement to the previous axiom, but not strictly necessary for what follows. The geometric significance of the generalisation from adicity to `weak' adicity is that we will be allowed to puncture formal schemes along closed, but not necessarily algebraic, formal subschemes. The only cost is that the completed localisations are not necessarily open homomorphisms; I do not know anything that would need such a hypothesis.\end{remark}

A pair with these additional data fixed is called a \emph{locally convex $\F_1$-algebra}. A ring homomorphism $A\rightarrow B$ preserving the ring of integers induces a pullback map on discs; this map is required to preserve the filtration (i.e. be continuous).

\

The structure sheaves in rigid geometry also have the following:
\begin{itemize}\item $A$ is a localisation of $A^+$. \hfill \emph{Tate}\end{itemize}
The nomenclature is an approximation to Huber's terminology `Tate ring' \cite{Hubook}. This condition is not preserved by limits, and so in sheaves of Tate rings the property holds only locally. We do at least have

\begin{lemma}\label{RIG_TATEPROD}Any product of Tate rings is Tate.\end{lemma}
\begin{proof}Let $(A_i=A_i^+[S_i^{-1}];A_i^+)$ be a family of Tate rings. Their product $\prod_iA_i$ is a localisation of $\prod_iA_i^+$ at all elements of the form $(1,\ldots,f,\ldots)$ with $f\in S_i$ in the $i$th position.\end{proof}

\begin{eg}\label{POLY_VAL}If $H$ is a totally ordered group, the special $\F_1$-algebra $\F_1[z^H]$ could be called the `equal characteristic one valuation field with value group $H$', and denoted accordingly $\F_1(\!(t^{-H})\!)$. This notation presupposes that we consider it equipped with the `valuation ring' $\F_1[\![t^{-H}]\!]$ associated to $H^\circ$, and the evident $H$-indexed filtration by $\F_1[\![t^{-H}]\!]$-submodules. The sign convention is such that $t^{-\lambda}$ converges to $0$ as $\lambda\rightarrow-\infty$ in $H$.

The point of the remarks of \S\ref{POLY_NONPOS} is that the $\F_1$-algebra $\sh O\{\Delta\}$ associated to $\Aff_\Delta(N,H)$ is Hausdorff with respect to the $t$-adic topology, that is, it is a Banach $\F_1(\!(t^H)\!)$-algebra.\end{eg}

\subsection{Affine presentation}\label{RIG_AFFINE}

The passage from the categorical to the algebraic picture of rigid analysis rests on a simple proposition:

\begin{prop}\label{RIG_SAT}Let $X$ be a rigid space with model $j:X\rightarrow X^+$. Then $j_*\sh O_X^+$ is the integral closure of the image of $\sh O_{X^+}$ inside $j_*\sh O_X$.\end{prop}

This is an immediate consequence of lemma I.\zref{I_RIG_THELEMMA} for $\F_1$ and I.\zref{I_A_PROJ_FINITE} for $\Z$. 

Let us denote by ${}_Z\mathbf{FSch}^\mathrm{aff}$ the full subcategory of ${}_Z\mathbf{FSch}$ whose objects $(X^+,Z)$ are such that $X^+$ is affine and $Z$ is a union of principal divisors. It generates ${}_Z\mathbf{FSch}^\mathrm{div}$, and since one can always admissibly blow up a marked formal scheme to get $Z$ a union of Cartier divisors, it is also a site for the rigid topos $\Sh\mathbf{Rig}$.

Also, for this section, $\mathrm{Alg}$ will denote the category of (not necessarily relatively normal) Tate algebras (\S\ref{RIG_LC}).\footnote{One could alternatively take for ${}_Z\mathbf{FSch}^\mathrm{aff}$ the category on which $Z$ is formed of Cartier divisors. This would correspond on the algebraic side to the assumption that $A^+\subseteq A$.}

Pushing forward the structure sheaf from $\mathbf{FSch}^\mathrm{aff}$, one obtains a sheaf $\sh O^+$ on ${}_Z\mathbf{FSch}^\mathrm{aff}$ of adic algebras. Let $S\subseteq\sh O^+$ denote the multiplicatively closed subsheaf of local sections $f\in\sh O^+$ such that $\sh O^+/f$ is supported on $Z$. Then $\sh O:=\sh O^+[S^{-1}]$ is a sheaf of Tate algebras on ${}_Z\mathbf{FSch}^\mathrm{aff}$ with ring of integers $\sh O^+$.

Let $(A;A^+)$ be a Tate algebra. The set of elements of $A^+$ that become invertible in $A$ define a collection of Cartier divisors on $\Spec A^+$, and hence an object of ${}_Z\mathbf{FSch}^\mathrm{aff}$. This construction is contravariantly functorial
\[ \Spec:\Alg\rightarrow{}_Z\mathbf{FSch}^\mathrm{aff} \]
and it is a spectrum functor in the sense of ringed toposes, that is, it is \emph{left adjoint} to the structure sheaf
\[ \sh O:{}_Z\mathbf{FSch}^\mathrm{aff} \rightarrow\Alg. \]
In fact, the two functors are inverse equivalences of categories.

\begin{defn}A rigid analytic space is said to be \emph{affine} if it admits a model in ${}_Z\mathbf{FSch}^\mathrm{aff}$. The full subcategory of $\mathbf{Rig}$ whose objects are affine is denoted $\mathbf{Rig}^\mathrm{aff}$.\end{defn}

The composite of the Spec functor with the localisation
\[ {}_Z\mathbf{FSch}^\mathrm{aff} \rightarrow \mathbf{Rig}^\mathrm{aff} \]
remains left adjoint to $\sh O$, but it is no longer an equivalence: the counit $A\rightarrow\sh O^+(\Spec A)$ need not be an isomorphism. To correct this, we must invert the morphisms in $\Alg$ dual to affine admissible blow-ups. 

We have identified these homomorphisms: by proposition \ref{RIG_SAT}, they are precisely the \emph{$Z$-admissible integral extensions}. In particular, $\sh O(\Spec A)$ is exactly the \emph{relative normalisation} of $A$. This is why we usually - indeed, henceforth - make the assumption that Tate rings are relatively normal.

\begin{prop}$\mathrm{Spec}$ and $\sh O$ form inverse anti-equivalences between $\mathbf{Rig}^\mathrm{aff}$ and the category of relatively normal Tate rings.\end{prop}

Dually, this result reflects the fact that the localisation on affine objects has a fully faithful right adjoint $\mathbf{Rig}^\mathrm{aff}\hookrightarrow{}_Z\mathbf{FSch}^\mathrm{aff}$. In other words,

\begin{cor}An affine rigid space $X$ has a \emph{unique} affine, relatively normal model given by the spectrum of $\sh O^+(X)$.\end{cor}

We conclude by describing the presentation of $\Sh\mathbf{Rig}$ by means of $\mathbf{Rig}^\mathrm{aff}$. An open immersion in $\mathbf{Rig}^\mathrm{aff}$ with target $\Spec A$ is given by the base change to $A$ of a principal affine subset of an admissible modification of $\Spec A^+$, which up to relative normalisation is just an admissible blow-up. It has co-ordinate algebra of the form
\[ A\rightarrow A\{T/s\} = \left(A\tens_{A^+}A^+\{T/s\};A^+\{T/s\} \right) \]
with $T$ of finite type, $s\in T$ and $TA=A$. Since $T/s$ is an open disc, $A\{T/s\}$ is a Banach module over $A$. 

Conversely, any \emph{completed localisation} of $A$ - that is, algebra of the form $A\{T/s\}$ with $T$ finitely generated - comes from an affine open immersion. In other words, a morphism $\Spec B\rightarrow \Spec A$ in $\mathbf{Rig}^\mathrm{aff}\subset\Alg^\mathrm{op}$ is an open immersion if and only if $B$ is the relative normalisation of a completed localisation of $A$.

The following are equivalent for a finite family of completed localisations $A\rightarrow A\{T_i/s_i\}$:
\begin{enumerate}
\item $\Spec A = \bigcup_i \Spec A\{T_i/s_i\}$;
\item there exist admissible blow-ups $X^+\rightarrow\Spec A^+$ and $U_i\rightarrow\Spec A^+\{T_i/s_i\}$  such that $U_i^+\hookrightarrow X^+$ and $X^+= \bigcup_iU_i^+$;
\item $A\rightarrow \prod_iA\{T_i/s_i\}$ is a universally effective monomorphism in $\Alg_{\F_1}$.
\end{enumerate}
These criteria can be used to present the rigid topos and $\mathbf{Rig}$ in terms of topological algebra.

\subsection{Summary of categories}

I introduced a lot of categories in this section, so let me draw most of the main players out here as what I hope to be a handy reference utility.
\[\xymatrix{
\mathbf{Sch}_{\F_1}^\mathrm{aff}\ar[r]\ar[dd] & \mathbf{FSch}^\mathrm{aff}_{\F_1}\ar[dd] & {}_Z\mathbf{FSch}^\mathrm{aff}_{\F_1} \ar[d]\ar[r]\ar[l] & \mathbf{Rig}^\mathrm{aff}_{\F_1}\ar[d] \\
&& {}_Z\mathbf{FSch}^\mathrm{qcqs}_{\F_1} \ar[d]\ar[r] & \mathbf{Rig}^\mathrm{qcqs}_{\F_1}\ar[d] \\
\mathbf{Sch}_{\F_1}\ar[r]\ar[dr] & \mathbf{FSch}_{\F_1}\ar[d] & {}_Z\mathbf{FSch}_{\F_1} \ar[d]\ar[r]\ar[l] & \mathbf{Rig}_{\F_1}\ar[d] \\
& \Sh\mathbf{FSch}_{\F_1} & \Sh{}_Z\mathbf{FSch}_{\F_1} \ar[r]^-{\zeta^*}\ar[l] & \Sh\mathbf{Rig}_{\F_1}
}\]
I remind the reader of the following:
\begin{itemize}\item The categories in the top row are opposite to the category of $\F_1$-algebras, pro-discrete $\F_1$-algebras, Tate $\F_1$-algebras, and relatively normal Tate $\F_1$-algebras, respectively.
\item The topoi in the bottom row are each generated by any category vertically above, \emph{except} for ${}_Z\mathbf{FSch}^\mathrm{aff}_{\F_1}\subset \Sh{}_Z\mathbf{FSch}_{\F_1}$ (\S\ref{RIG_AFFINE}). The arrows in the bottom row are the pullback functors for geometric morphisms.
\item $\Sh\mathbf{FSch}_{\F_1}$ is actually a presheaf category on $\mathbf{FSch}^\mathrm{aff}_{\F_1}$.
\item The horizontal arrow in the top-right and the one immediately below are left exact categorical localisations.
\end{itemize}
We also have used intermediate marking categories \[{}_Z\mathbf{FSch}^\nu_{\F_1}\subset{}_Z\mathbf{FSch}^\mathrm{inv}_{\F_1}\subset{}_Z\mathbf{FSch}_{\F_1}^\mathrm{div}\subset{}_Z\mathbf{FSch}_{\F_1}\] whose markings are relatively normal Cartier, Cartier, and divisorial, respectively.

\subsection{Riemann-Zariski space}\label{RIG_RZ}

We have constructed the category of qcqs rigid analytic spaces as a subcategory of the category of pro-formal schemes; by definition, the rigid space associated to a marked rigid space $(X^+,Z)$ is a formal limit \[ X=\lim \widetilde X^+ \] over admissible blow-ups $\widetilde X^+$ of $X^+$. By understanding this limit instead in the category of locally linearly-topologised-monoidal topological spaces, we can define a space $\mathrm{RZ}(X)$, the \emph{Riemann-Zariski space} of $X$ (or $(X^+,Z)$). The structure of the terms $\widetilde X^+$ as objects of ${}_Z\mathbf{FSch}$ equips the limit with a sheaf $(\sh O_X;\sh O_X^+)$ of Tate $\F_1$-algebras (though recall that the Tate condition does not persist on the spaces of sections over large open sets of $\mathrm{RZ}(X)$).

The topology on $\mathrm{RZ}(X)$ is that of the limit; it is therefore generated by the quasi-compact open sets of models $X^+$, with two such sets being equal if and only if they agree on a coinitial family of models. It is equal to the lattice of open subobjects of $X$ in $\Sh\mathbf{Rig}$. In particular, $\mathrm{RZ}(X)$ is a quasi-compact and quasi-separated topological space.

It follows that an open immersion $U\hookrightarrow X$ of (qcqs) rigid spaces gives rise to an open immersion of their Riemann-Zariski space, and hence that this construction can be globalised to a functor
\[ \mathbf{Rig}_{\F_1}\rightarrow\mathbf{Top} \]
into the category of topological spaces. As usual, this functor can be upgraded to take values in the category of topological spaces equipped with a sheaf of Tate $\F_1$-algebras. We will usually not distinguish between a rigid analytic space and its Riemann-Zariski space.

One can give a point-set-topological definition of rigid spaces in this manner.

\begin{defn}[Rigid spaces as Tate-ringed topological spaces]Let $(X^+,Z)$ be an object of ${}_Z\mathbf{FSch}_{\F_1}$, that is, a formal scheme $X^+\in\mathbf{FSch}_{\F_1}$ equipped with a closed subset $Z$. Define a Tate-ringed topological space $\mathrm{RZ}(X^+,Z)$ by the preceding formula, limiting over all blow-ups of $X^+$ along $Z$.

A \emph{rigid analytic space} is a Tate-ringed topological space $X$ locally modelled by $\mathrm{RZ}(X^+,Z)$ for some $(X^+,Z)\in {}_Z\mathbf{FSch}_{\F_1}^\mathrm{qcqs}$. Note that on sufficiently small open sets $U\subseteq X$, one always has a canonical reprentative $\Spec\sh O(X)$.\end{defn}

It is quite likely that the underlying set of a Riemann-Zariski space can be described in terms of valuations; for a special case, see the appendix \ref{APPEND}.

\section{Overconvergence in rigid analytic geometry}\label{SEP}

In \cite[\S\zref{I_SEP}]{part1}, we defined notions of overconvergence via the existence of solutions to certain \emph{extension problems}. In applying this to rigid geometry, we should keep in mind that the primary desired property for proper morphisms of rigid spaces is that they should admit proper models; this is proposition \ref{SEP_PROPER_MODEL}.

An alternative approach would be to simply define proper morphisms to be those admitting a proper model. This would remove some of the difficulties encountered below. At least for paracompact morphisms, it would even be possible to extend this approach to defining overconvergence; cf. \ref{SEP_SUR_MODEL}.





\subsection{Rigid analytic closure operator}\label{SEP_CLOSURE}
To avoid constantly having to pick models, it will be convenient to be able to pass to the \emph{formally embedded closure} at the level of rigid spaces. This is a simple special case of the construction in \S I.\zref{I_SEP_NHOOD} of the overconvergent germ with respect to $\P$ the class of formal embeddings. This class evidently satisfies (P1-4), and (SC) is trivial since the property of being a formal embedding is local on the target.

Let $U\hookrightarrow V$ be an open immersion in $\mathbf{Rig}$, and suppose for now that $V$ is paracompact and, in the $\F_1$ case, that $U/V$ is modelled by an \emph{affine} open immersion $U^+\hookrightarrow V^+$ of formal schemes. Then
\[ \mathrm{cl}(U^+/V^+)\rigless Z\hookrightarrow V \]
is formally embedded. If $\tilde V^+\rightarrow V^+$ is a $Z$-admissible modification with base change $\tilde U^+$ to $U^+$, then
\[ \mathrm{cl}(\tilde U^+/\tilde V^+)\hookrightarrow\mathrm{cl}(U^+/V^+)\times_{V^+}\tilde V^+ \]
is an affine formal embedding. 

By ranging over models $U^+\subseteq V^+$ of $U/V$, we may therefore define a \emph{rigid analytic closure} as a pro-object
\[\mathrm{cl}(U/V)=\lim_{U^+/V^+}\mathrm{cl}(U^+/V^+)\rigless Z \in\Pro(\mathbf{Rig}_V)\]
in the category of rigid spaces over $V$. 

\begin{remark}Since the transition maps are modelled by affine morphisms of formal schemes, this pro-object can actually be realised as a rigid analytic space over $V$, though it is not formally embedded.\end{remark}

In other words, $\mathrm{cl}(U/V)$ is actually a sheaf of pro-objects on $\sh U^\mathrm{qcqs}_{/V}$, which we confuse with its global sections. We can use this fact to extend the definition to arbitrary $V$. 

For general $V$, we cannot guarantee that $\mathrm{cl}(U/V)$ is contained in any formally embedded subspace. Despite this, the statement $\mathrm{cl}(U/V)\subseteq Z$ is still well-formed for any immersed $Z\hookrightarrow X$; in particular, for open subsets.


\subsection{Extension problems for rigid spaces}\label{SEP_EXT}
We define $f\P$, $f\!i/\P$ in the topos of marked formal schemes to be the classes of morphisms represented by formally projective, resp. formally integral/projective, morphisms. The forgetful functor
\[  \Sh{}_Z\mathbf{FSch}\hookrightarrow \Sh\mathbf{FSch} \]
both preserves and detects overconvergence. It follows that both its adjoints do as well. The proof being more a problem of notation than anything else, I omit it. 

The essential image of the classes $f\P$, $f\!i/\P$, in the rigid topos are also called formally projective (resp. integral/projective) morphisms, cf. def. \ref{RIG_REPS}. Since admissible modifications are formally integral/projective, the existence of one $f\!i/\P$ model for a morphism implies that of a coinitial family (locally on the target).


\begin{lemma}\label{SEP_RIG_TO_FSCH}The localisation
\[\Sh{}_Z\mathbf{FSch}\rightarrow\Sh\mathbf{Rig}. \]
preserves overconvergence.\end{lemma}
\begin{proof}Let $V$ be any qcqs rigid space, $U\hookrightarrow V$ an open immersion. Then
\[ \mathrm{Sur}_{U/V}\rightarrow\mathrm{Sur}_{U^+/V^+}\rigless Z \]
is, as a pro-object of $\Sh\mathbf{Rig}$, a cofiltered limit over models $U^+/V^+$ of $U/V$. It follows that for a marked formal scheme $(S^+,Z)$ and structural map $V\rightarrow S:=S^+\rigless Z$,
\[ \colim_{V^+}\Hom_{S^+}(\mathrm{Sur}_{U^+/V^+},X^+)\tilde\rightarrow\Hom_{S}(\mathrm{Sur}_{U/V},X^+\rigless Z) \]
for any marked formal scheme $X^+/S^+$. We have shown criterion \emph{iv)} of lemma I.\zref{I_SEP_COMPARE}.\end{proof}

For the rigid analytic puncturing to \emph{detect} overconvergence, we will certainly need at least to restrict to the category ${}_Z\mathbf{FSch}^\mathrm{inv}$ of formal schemes marked along Cartier divisors. Under this restriction, we may try to apply criterion \emph{iii)} of lemma I.\zref{I_SEP_COMPARE} by showing that the square
\[\xymatrix{ \Hom_{S^+}(\mathrm{Sur}_{U^+/V^+},-)\ar[r]\ar[d] & \Hom_S(\mathrm{Sur}_{U/V},-) \ar[d] \\ \Hom_{S^+}(U^+,-)\ar[r] & \Hom_S(U,-) }\]
is Cartesian for any $U^+/V^+/S^+$.

The problem is that the existence of a lift $\mathrm{Sur}_{U/V}\rightarrow X$ of a morphism $U\rightarrow X$ only, \emph{a priori}, guarantees that we can find a morphism into $X^+$ from a modification $\tilde V^+\rightarrow V^+$ restricting to an \emph{admissible} modification $\tilde U^+$ of $U^+$:\[\xymatrix{ \widetilde U^+\ar[r]\ar[d]_{Z-\text{adm.}} & \widetilde V^+\ar@{-->}[ddl]\ar[d]^{f\P}  \\ U^+\ar[r]\ar[d] & V^+\ar[d] \\ X^+\ar[r]^f & S^+ }\] Typically $\widetilde V^+$ will not admit a section over $U^+$. 

We would like, nonetheless, for these data to give rise to a unique morphism $\mathrm{Sur}_{U^+/V^+}\rightarrow X^+$ over $S^+$. In other words, we want that $\mathrm{Sur}_{\widetilde U^+/\widetilde V^+}$ and $U^+$ form a \emph{canonical covering} of $\mathrm{Sur}_{U^+/V^+}$ in the category of algebraic spaces.

\begin{lemma}\label{SEP_PUSHOUT}Let $\widetilde V\rightarrow V$ be surjective. The square
\[\xymatrix{ \widetilde U \ar[r]\ar[d] & \mathrm{Sur}_{\widetilde U/\widetilde V}\ar[d] \\ 
U\ar[r] & \mathrm{Sur}_{U/V}}\]
is a pushout in the category of formal algebraic spaces separated and locally of finite type over $S$.\end{lemma}
\begin{proof}[Proof over $\Z$]Suppose we are given a square
 \[\xymatrix{ \widetilde U \ar[r]\ar[d] & \widetilde V\ar[d] \\ U\ar[r] & X}\]
and, without loss of generality, that $U$ is dense in $V$. We may assume that all players are schemes; the general result follows from passing to an inductive colimit. In particular, $\widetilde V\rightarrow V$ is projective. We suppress $S$ from the notation. 

We may obtain a morphism into $X$ from the closure $\tilde V$ of $U$ in $X\times V$. Since $X$ is separated and locally of finite type, $\tilde V$ is separated and of finite type over $V$. It will be enough to show that $\tilde V\rightarrow V$ is proper, and hence that $\mathrm{Sur}_{U/V}\rightarrow \tilde V$. 

Because $\widetilde U$ is dense in $\widetilde V$, the section $\widetilde V\rightarrow X\times V$ factors through $\tilde V$. By \cite[ \textbf{II}.5.4.3.\emph{i)}]{EGA}, $\widetilde V\rightarrow\tilde V$ is proper. In particular, it is closed; being surjective over a dense open subset, it is therefore surjective. Finally, by \cite[ \textbf{II}.5.4.3.\emph{ii)}]{EGA}, $\tilde V\rightarrow V$ is proper.\end{proof}

Essentially the same argument would carry through in the $\F_1$ case if we had the analogue of \cite[ \textbf{II}.5.4.3.\emph{ii)}]{EGA} (cf. lemma I.\zref{I_phlegm}) for proper morphisms over $\F_1$ and if we knew that torsion-free modifications were `strongly surjective' in the sense of \emph{loc. cit.}.

Unfortunately, we don't have either of these things, and so we are reduced to pursuing a more ad hoc approach using the combinatorial methods of \S I.\zref{I_PFAN_SUR}:

\begin{proof}[Proof over $\F_1$]Let us fix a square 
\[\xymatrix{ \widetilde U \ar[r]\ar[d] & \widetilde V\ar[d] \\ U\ar[r] & X}\]
and assume that all players are schemes and that $U/V$ is affine and dense.

\vspace{5pt}
\emph{Toric case.} Suppose that $X/S$ lives in $\mathbf{Sch}_{\F_1}^\mathrm{qi/nb}$ and that $S$ is Noetherian. We may as well assume that $X$ is connected; it is then represented by a fan $\Sigma_X\subset N_X$. The morphism $\sigma_U\rightarrow\Sigma_X$ induces a map $\sigma_V(\R)\rightarrow N_X(\R)$ whose image, because $\sigma_{\widetilde V}\rightarrow\sigma_V$ is bijective, is in the support of $\Sigma_X$. We therefore obtain a morphism $\tilde \sigma_V\rightarrow\Sigma_X$ from a subdivision $\tilde\sigma_V$ of $\sigma_V$ not touching $\sigma_U$.

\vspace{5pt}
\emph{Noetherian case.} Suppose that $S$ is Noetherian. Then $V$ and $\widetilde V$ are Noetherian and hence can be decomposed into finitely many quasi-integral closed subschemes $V_i,\widetilde V_i$ which, for simplicity, we index by the same set.

By replacing $X$ with the closure $X_i$ of the image of $U_i$, we reduce to the toric case. We therefore have unique extensions $\mathrm{Sur}_{U_i/V_i}\rightarrow X$. By proposition I.\zref{I_INT_DECOMP}, these extensions glue to a unique extension $\mathrm{Sur}_{U/V}\rightarrow X$.

\vspace{5pt}
\emph{General case.} Let us replace $X$ with a quasi-compact open subset through which $\widetilde V\rightarrow X$ factors. All objects are now of finite type over $S$, and so there exists a Noetherian formal scheme $S_0$, a diagram \[\xymatrix{ \widetilde U_0\ar[r]\ar[d]_{Z-\text{adm.}} & \widetilde V_0\ar@{-->}[ddl]\ar[d]^\P  \\ U_0\ar[r]\ar[d] & V_0\ar[d] \\ X_0\ar[r]^f & S_0 }\] over $S_0$ and an embedding $X\hookrightarrow X_0\times_{S_0}S$. By the Noetherian case, there is a unique map $\mathrm{Sur}_{U_0/V_0}\rightarrow X_0$ descending from $\mathrm{Sur}_{\widetilde U_0/\widetilde V_0}$, and hence an extension
\[ \mathrm{Sur}_{U/V}\rightarrow\mathrm{Sur}_{U_0/V_0}\times_{S_0}S\rightarrow X_0\times_{S_0}S. \]
Since embeddings are proper (proposition I.\zref{I_SEP_EMBEDDING}), the extension actually factors through the embedded formal subscheme $X$.\end{proof}

\begin{remark}[Descent along blow-downs]In certain cases, it may even be possible to construct a pushout of $\widetilde U^+\rightarrow\widetilde V^+$ along the blow-down $\widetilde U^+\rightarrow U^+$; for instance, this is what is happening in proposition I.\zref{I_INT_DECOMP}.

Here is a sketch of an algorithmic construction in the \emph{toric} case, that is, when the square lives in $\mathbf{FSch}_{\F_1}^\mathrm{n/nb}$. We may assume that all players are Noetherian. Then $\Sigma_{\tilde V^+}\rightarrow\Sigma_{V^+}$ is a subdivision of a neighbourhood of $\Sigma_{U^+}$. Let us consider these as fans immersed in $N_{U^+}$. The objective is to `coarsen' $\Sigma_{\tilde V^+}$ in a minimal way such that it no longer subdivides $\Sigma_{U^+}$.

Let us begin by deleting the cones of $\Sigma_{\tilde V^+}$ that intersect $\Sigma_{U^+}$ in some set that is not a face. The resulting collection of cones may fail to be strongly convex, so we continue by deleting cones where this failure occurs. Since $\Sigma_{\tilde V^+}$ has finitely many cones, this procedure terminates after a finite number of steps, and its termination implies that we are left with a punctured fan. This object is the blow-down.\end{remark}

\subsection{Separation}\label{SEP_SEP}

\begin{defn}A morphism of rigid spaces is said to be \emph{locally separated} if it is locally $f\P$-separated (\cite[def. \zref{I_SEP_EXTENS_DEF}]{part1}). It is \emph{separated} if it is quasi-separated and locally separated.\end{defn}

By lemma \ref{SEP_RIG_TO_FSCH}, if $X^+\rightarrow S^+$ is a (locally) separated morphism of formal schemes, then $X^+\rigless Z\rightarrow S^+\rigless Z$ is (locally) separated for any marking $Z$ of $X^+$ and $S^+$. Conversely:

\begin{lemma}\label{SEP_FSCH_TO_RIG}Let $X^+\rightarrow S^+$ be a morphism of formal schemes with Cartier marking.

 If $X^+\rigless Z\rightarrow S^+\rigless Z$ is locally separated, then $X^+/S^+$ is locally separated.\end{lemma}
\begin{proof}Let \[\xymatrix{ U^+\ar[r]\ar[d] & V^+\ar@<2pt>[dl]\ar@<-2pt>[dl] \\ X^+}\] be an extension problem with two solutions. Since $X$ is locally separated, the two morphisms become equal after a $Z$-admissible modification $\tilde V^+$ of $V^+$.

As $Z$ is invertible, $\tilde V^+\rightarrow V^+$ is surjective, and so $V^+\rightrightarrows X^+$ both have the same underlying map of sets. Assuming $V^+$ affine, we may therefore replace $X$ with an affine open subset through which both solutions factor. But $U^+\rightarrow V^+$ is an epimorphism in the category of affine schemes, so the two arrows are equal.\end{proof}

\begin{prop}\label{SEP_SEP_MODEL}Let $X\rightarrow S$ be a quasi-separated morphism of analytic spaces. The following are equivalent:
\begin{enumerate}\item $X/S$ is separated;
\item every quasi-compact open subset of $X$ is separated over $S$;
\item locally on $S$, every quasi-compact open subset of $X$ admits a separated model;
\item every model with Cartier marking of every open subset of $X$ is separated;
\item the diagonal of $X/S$ is an embedding.\end{enumerate}\end{prop}
\begin{proof}The equivalence \emph{i)}$\Leftrightarrow$\emph{ii)} is elementary. We have just seen, through lemmas \ref{SEP_RIG_TO_FSCH} and \ref{SEP_FSCH_TO_RIG}, the equivalence with \emph{iii)} and \emph{iv)}. Finally, \emph{v)} is a consequence of corollary I.\zref{I_SEP_CRITERIA} applied to local models of the diagonal, which is qcqs by hypothesis.\end{proof}

\subsection{Overconvergence and propriety}\label{SEP_PROPER}

\begin{defn}\label{SEP_DEF_PROPER}A morphism $X\rightarrow S$ is said to be \emph{overconvergent}, resp. \emph{proper}, if it is $f\P$-overconvergent, resp. $f\P$-proper (def. I.\zref{I_SEP_EXTENS_DEF}) and locally of finite type.\end{defn}

By lemma \ref{SEP_RIG_TO_FSCH}, if $X\rightarrow S$ is an overconvergent (resp. proper) morphism of formal schemes, then $X\rigless Z\rightarrow S\rigless Z$ is overconvergent (resp. proper) for any marking of $X$ and $S$. 

The converse is a little more difficult:


\begin{prop}\label{SEP_PROPER_MODEL}Let $f:X\rightarrow S$ be a morphism of analytic spaces. The following are equivalent:
\begin{enumerate}\item $f$ is proper;
\item locally on $S$, $f$ admits a proper model;
\item a model with Cartier marking of a base change of $f$ is proper.
\end{enumerate}\end{prop}
\begin{proof}A proper morphism is qcqs and so certainly admits a model locally on the base. The trick is to show that this model is proper whenever the marking is Cartier. By proposition \ref{SEP_SEP_MODEL}, it is at least separated. But then the existence of solutions to extension problems follows from lemma \ref{SEP_PUSHOUT}.\end{proof}

By choosing models, it follows immediately from proposition I.\zref{I_SEP_SUR_MODELS}:

\begin{cor}\label{SEP_SUR_MODEL}
If $f$ is paracompact, then the following are equivalent:
\begin{enumerate}\item $f$ is overconvergent;
\item every proper $X$-space qcqs over $S$ is proper over $S$;
\item every formally embedded subspace of $X$ qcqs over $S$ is proper over $S$.
\end{enumerate}\end{cor}

\subsection{Overconvergence apr\`es Deligne}\label{SEP_DELIGNE}

In \cite{Deligne}, Deligne defines a sheaf $F$ on the small topos $\Sh(X)$ of a quasi-separated rigid analytic space $X$ over $\Z$ to be overconvergent if and only if, for all qcqs open $U\subseteq X$,
\[ \colim_{U\subseteq\mathrm{cl}(U/X)\subseteq V}F(V)\rightarrow F(U) \]
is an isomorphism.\footnote{Actually, Deligne works only with rigid analytic spaces admitting a Noetherian formal model, but the definition works without modification in an arbitrary quasi-separated geometry.} The term on the left does not depend on whether we interpret $\mathrm{cl}(U/X)$ in terms of point set topology or as a pro-object as in \S\ref{SEP_CLOSURE}. As such, the definition is equally well-formed over $\F_1$.

By definition of the pullback from $\Sh(X)$ to $\Sh\mathbf{Rig}$, this term is simply the value of $F$ on the pro-object $\mathrm{cl}(U/X)$. Thus the definition equivalently says that every diagram
\[\xymatrix{ U\ar[r]\ar[d] & \mathrm{cl}(U/X)\ar[r]\ar@{-->}[dl] & X\ar@{=}[d] \\ F\ar[rr] && X }\]
has a unique extension $\mathrm{cl}(U/X)\rightarrow F$.

This is a particular case of an extension problem with $\P$ the class of formal closed embeddings. Since a morphism $U\rightarrow F$ from an arbitrary object of $\Sh\mathbf{Rig}_X$ locally factors through an open subset of $X$, it follows that we have unique solutions for any extension problem $U/V$. Thus we arrive at the following, equivalent for quasi-separated $X$, formulation of Deligne's definition:

\begin{defn}A small sheaf on a rigid analytic space $X$ (over $\F_1$ or $\Z$)  is \emph{Deligne overconvergent} if it is $\P$-overconvergent as an object of $\Sh\mathbf{Rig}_X$ with $\P$ the class of formal embeddings.\end{defn}

Deligne overconvergence implies overconvergence in the sense of definition \ref{SEP_DEF_PROPER}. For the converse statement, we have the following (compare lemma I.\zref{I_SEP_EMBED_PROPER}):

\begin{lemma}Let $U\hookrightarrow V$ a quasi-compact open immersion of $X$-spaces. The natural map \[\Hom_X(\mathrm{cl}(U/V),-) \tilde\rightarrow \Hom_X(\mathrm{Sur}_{U/V},-)\] is an isomorphism of functors on $\Sh(X)$.\end{lemma}
\begin{proof}Let $F\in\Sh(X)$ and $\mathrm{Sur}_{U/V}\rightarrow F$. By definition,
\[ F(\mathrm{Sur}_{U/V})\cong\colim_{\tilde V\rightarrow V}\colim_{\tilde V\rightarrow W\subseteq X}F(W) \]
locally on $V$; so any $f\in F(\mathrm{Sur}_{U/V})$ is represented by a section $W\rightarrow F$ such that $\tilde V\rightarrow W$ for some formally projective modification $\tilde V\rightarrow V$. We will show that this implies $V\rightarrow W$.

Without loss of generality, suppose $V$ is affine, and that we have models
\[\xymatrix{ &\tilde V^+\ar[d] \\ U^+\ar[r]\ar[ur] & V^+ }\]
such that $U^+\subseteq V^+$ is dense. By compactness we may also assume that $W$ is qcqs, and thus that $W\times_XV\hookrightarrow V$ has a model $\tilde V^+\rightarrow W^+\subseteq V^+$. 

But $\tilde V^+\rightarrow V^+$ is surjective, and so $W^+=V^+$.\end{proof}

\begin{prop}A sheaf on the small site of a rigid analytic space is overconvergent if and only if it is Deligne overconvergent.\end{prop}

\begin{cor}\label{SEP_SUR_OPEN}Let $U\subseteq X$ be an open subset of a quasi-separated analytic space $X$. The following are equivalent:
\begin{enumerate}\item $U\subseteq X$ is overconvergent;
\item for any $W\subseteq U$ affine over $X$, $U$ contains $\mathrm{cl}(W/X)$;
\item for any $W\subseteq U$ quasi-compact over $X$, $U$ contains $\mathrm{cl}(W/X)$.\end{enumerate}\end{cor}

\subsection{Overconvergent site}\label{SEP_TOP}
The idea of recovering Berkovich's Hausdorff topology as a `coarsening' of the topology of rigid analytic spaces is apparently also due to Deligne \cite{Deligne}. I learned about it from \cite[\S8]{Hubook}, where it appeared under the name `partially proper' topology.

The generalities in this paragraph make sense for any spatial geometric context with a notion of overconvergence generated by a class of morphisms $\P$ as in \S I.\zref{I_SEP}; in particular, we will apply it to the topos of \emph{collages} below in \S\ref{AFF}. For simplicity and concreteness, we restrict attention here to the motivating case of rigid analytic spaces.


\paragraph{Large site} Let $X$ be a rigid analytic space over $\F_1$ or $\Z$, and let $\Sh\mathbf{Rig}^\mathrm{sur}_X$ denote the \emph{overconvergent topos} of $X$, that is, the category of overconvergent sheaves over $X$. By part \emph{v)} of proposition I.\zref{I_SEP_STABILITY}, every morphism in $\Sh\mathbf{Rig}^\mathrm{sur}_X$ is overconvergent. The category $\mathbf{Rig}^\mathrm{sur}_X$ of rigid analytic spaces overconvergent over $X$ is a full subcategory.

\begin{lemma}\label{hat}$\mathbf{Rig}^\mathrm{sur}_X$ is a spatial site (def. I.\zref{I_TOPOS_DEF}) for $\Sh\mathbf{Rig}^\mathrm{sur}_X$.\end{lemma}
\begin{proof}Since overconvergence is local on $X$, we may assume that $X$ is quasi-separated. We first show that the overconvergent topos is generated by analytic spaces formally \emph{projective} over $X$.  Since every object $F\in\Sh\mathbf{Rig}^\mathrm{sur}_X$ is a colimit in $\Sh\mathbf{Rig}_X$ of affine rigid analytic spaces of finite type over $X$, it will be enough to show that every morphism $U\rightarrow F$ from such a space $U$ factors through some object formally projective over $X$.

Since $U$ is of finite type over $X$, it may be embedded in some projective bundle $\P(\sh E)$. Overconvergence means that after replacing the latter with a formally projective modification, $\P(\sh E)\rightarrow F$. The problem is that $\P(\sh E)\rightarrow X$ may no longer be of finite type.

Applying the construction of expanded degenerations (I.\zref{I_SEP_EXP_DEG}) to the data $(U,\P(\sh E),Z=\text{reduction of }\P(\sh E))$, we obtain a morphism $U\subseteq U^\mathrm{\'el}\rightarrow F$. By proposition I.\zref{I_EXPANDED_DEGENERATION}, $U^\mathrm{\'el}\rightarrow X$ is overconvergent.\end{proof}

In other words, for any $X$, \[\mathbf{Rig}_X^\mathrm{sur}\subseteq\Sh\mathbf{Rig}^\mathrm{sur}_X\] is a spatial geometric context as in definition I.\zref{I_TOPOS_DEF} whose category of locally representable objects is the \emph{large overconvergent site} $\mathbf{Rig}_X^\mathrm{sur}$ of $X$.

Overconvergence being stable for composition and base change, a morphism $f:X\rightarrow Y$ naturally induces an essential spatial geometric morphism
\[ f:\Sh\mathbf{Rig}_X^\mathrm{sur}\rightarrow \Sh\mathbf{Rig}_Y^\mathrm{sur} \qquad  f_!:\mathbf{Rig}_X^\mathrm{sur}\leftrightarrows \mathbf{Rig}_Y^\mathrm{sur}:f^*\]
which, by locality on the base, makes $\Sh\mathbf{Rig}_-^\mathrm{sur}$ and $\mathbf{Rig}_-^\mathrm{sur}$ into \emph{stacks} on $\mathbf{Rig}$; the latter is locally a site for the former.

\paragraph{Small site} The terminal object $X$ of $\mathbf{Rig}_X^\mathrm{sur}$ has its own small topos $\Sh(X^\mathrm{sur})=\Sh(\sh U_{/X}^\mathrm{sur})$ which, although incoherent, is a subtopos of $\Sh(X)$ and therefore has enough points. In particular, it is \emph{spatial}, with determining sober topological space (or, if you prefer, locale) $X^\mathrm{sur}$. The geometric morphism $\Sh(X)\rightarrow\Sh(X^\mathrm{sur})$ induces a surjective continuous mapping
\[ b:X\rightarrow X^\mathrm{sur}\]
which the literature has called the \emph{separation map} (\cite[I.2.4.(c)]{FujiKato}). A morphism $f:X\rightarrow Y$ yields, by restriction from the large overconvergent topos, a natural geometric morphism $X^\mathrm{sur}\rightarrow Y^\mathrm{sur}$ such that the square
\[\xymatrix{ X\ar[r]^f\ar[d]_b & Y\ar[d]^b \\ X^\mathrm{sur}\ar[r]^f & Y^\mathrm{sur} }\]
commutes.

By definition, the square
\[\xymatrix{ \sh U_{/X}^\mathrm{sur} \ar[r]\ar[d]_{b^{-1}} & \mathbf{Rig}^\mathrm{sur}_X\ar[d] \\ \sh U_{/X}\ar[r] & \mathbf{Rig}_X }\]
is Cartesian. We would also like to know when the extended square
\[\xymatrix{ \Sh X^\mathrm{sur} \ar[r]\ar[d]_{b^*} & \Sh\mathbf{Rig}^\mathrm{sur}_X\ar[d] \\ \Sh X\ar[r] & \Sh\mathbf{Rig}_X }\]
is also Cartesian, so that every sheaf on the small site of $X$ that is overconvergent as an object of $\Sh\mathbf{Rig}_X$ is actually the pullback under $b$ of a sheaf on $X^\mathrm{sur}$.

\begin{lemma}\label{valid only when x is quasi-separated}Suppose that $X$ is purely analytic. Then $\Sh(X^\mathrm{sur})$ is the full subcategory of $\Sh(X)$ whose objects are overconvergent.\end{lemma}
\begin{proof}Follows as in the proof of lemma \ref{hat}, with $U\subseteq X$ a quasi-compact open subset. Since $X$ is purely analytic, $U^\mathrm{\'el}\rightarrow X$ is an open immersion.
\end{proof}

\begin{eg}This statement is false for formal schemes. For instance, any non-simply-connected Noetherian formal scheme has finite covering spaces, but no non-trivial overconvergent open subsets over which to find a section.\end{eg}

\subsection{Local compactness}

Overconvergent open sets are almost never quasi-compact. For practical reasons, it is often easier to work with quasi-compact objects; hence the following definition:

\begin{defns}[Overconvergent coverings]\label{SEP_SUR_COVER}A covering in $\mathbf{Rig}_X$ is \emph{overconvergent} if it can be refined by a covering in $\mathbf{Rig}^\mathrm{sur}_X$.

\label{SEP_COMPACT}An analytic space $X$ is said to be \emph{overconvergent-locally compact} if every qcqs open subset of $X$ admits a qcqs overconvergent neighbourhood. It is \emph{(overconvergent-)locally convex} if the neighbourhood can always be taken \emph{affine}. We usually abuse notation by omitting the prefix `overconvergent'.\end{defns}

Locally compact analytic spaces over $\Z$ were called variously \emph{locally quasi-compact} and \emph{strongly locally compact} in \cite[II.4.4.1]{FujiKato}; locally convex spaces are what Berkovich calls \emph{good} \cite{Berketale}. 

Any qcqs space is locally compact, and any affine analytic space is locally convex. Local convexity is usually only a reasonable condition in the non-affine case when the overconvergent topology is reasonable, that is, when the space is purely analytic.


\begin{lemma}\label{SEP_PARA=>LOCALLY}A paracompact analytic space is locally compact.\end{lemma}
\begin{proof}By choosing a model we may reduce to the case of formal schemes. Let $X$ be a paracompact formal scheme, $U_0\subseteq X$ be an affine open immersion. Since $X$ is paracompact, $\mathrm{cl}(U_0/X)$ is quasi-compact.\end{proof}

\begin{lemma}\label{SEP_SUR=>LOCALLY}Let $S$ be qcqs, $X$ overconvergent over $S$. Then $X$ is locally compact.\end{lemma}
\begin{proof}Let $U\hookrightarrow X$ be an affine open subset. Then $U\rightarrow S$ is of finite type, and so may be immersed into some projective bundle $\P(\sh E)/S$. By overconvergence, there is a formally projective morphism $\tilde\P(\sh E)\rightarrow \P(\sh E)$ and an extension
\[\xymatrix{ U\ar[r]\ar[d] & \tilde\P(\sh E)\ar[d]\ar@{-->}[dl] \\ X\ar[r] & S }\]
and by part \emph{v)} of proposition I.\zref{I_SEP_STABILITY}, $\tilde\P(\sh E)\rightarrow X$ is $f\P$-proper, and hence an overconvergent neighbourhood of $U$.\end{proof}

\begin{prop}Let $X\in\mathbf{Rig}$ be purely analytic and locally compact. Then:
\begin{enumerate}\item $X$ admits an overconvergent covering by qcqs open sets;
\item $X$ admits a covering by quasi-separated overconvergent open sets.\end{enumerate}
In particular, a morphism $X\rightarrow S$ is overconvergent if and only if it is overconvergent on every quasi-separated overconvergent open subset.\end{prop}
\begin{proof}First note that \emph{i)} immediately implies \emph{ii)} by the definition of overconvergent cover and the fact that any open subset of a quasi-separated space is quasi-separated. The proof of \emph{i)} follows from the construction of expanded degenerations (I.\zref{I_SEP_EXP_DEG}).\end{proof}

\begin{remark}Via the theory of collages exhibited in \S\ref{AFF_COLL}, it is easy enough to find examples of analytic spaces that are locally of finite type over $\F_1(\!(t)\!)$ but not locally compact; indeed, any non-locally-compact subset of $\R^n$ that is exhausted by rational polyhedra will do the trick.

In fact, this type of approach seems to permit the construction of toric analytic spaces with any egregious topological property imaginable.
\end{remark}

\subsection{Hausdorff quotient}
Under a certain technical assumption on an analytic space $X$, proposition \ref{SEP_SUR_OPEN} has point-set-topological consequences.
\begin{itemize}\item[(CL)] The formally embedded closure of a quasi-compact open immersion into $X$ is closed.\end{itemize}
It is equivalent that this be true for a coinitial family of models of $X$, since $X$ is weakly topologised with respect to such a family. Condition (CL) is manifestly satisfied by all analytic spaces over $\Z$ and \emph{quasi-integral} analytic spaces over $\F_1$.

\begin{remark}Let $A\rightarrow A[f^{-1}]$ be a localisation of discrete $\F_1$-algebras. The precise condition that $A,f$ must satisfy for the embedded closure of $\Spec A[f^{-1}]$ to be closed in $\Spec A$ is that the action of $f$ on $A/\mathrm{Ann}(f)$ be \emph{injective}. Therefore, for example, any $\F_1$-algebra with non-trivial idempotents will fail our condition.\end{remark}

Under condition (CL) and by corollary \ref{SEP_SUR_OPEN}, an open subset $U$ of $X$ is overconvergent if and only if, for every quasi-compact open immersion factoring through $U$, the point set topological closure is also in $U$. 

If $X$ is quasi-separated, then by \cite[I.2.3.4]{FujiKato} it is enough that the closure of every \emph{point} of $U$ is in $U$; thus, our definition is equivalent to \emph{op. cit.} I.2.4.10. In the case of purely analytic spaces (def. \ref{RIG_PURELY}), the arguments of \emph{op. cit.} \S II.4.1 apply to show that we are in the context of what the authors call `valuative' spaces. We may therefore apply the results of \emph{op. cit.} \S I.2.4.(d).

\begin{remark}If $X$ is a locally Noetherian topological space, then $X^\mathrm{sur}$ is just a single point. This includes most interesting formal schemes. A similar statement applies to any connected rigid space containing a formal scheme. As such, the overconvergent topology is only likely to be interesting for purely analytic rigid spaces (def. \ref{RIG_PURELY}).\end{remark}

\begin{thm}[Properties of the overconvergent topology]\label{SEP_SUR}Let $X$ be a rigid analytic space satisfying the condition $\mathrm{(CL)}$ - for instance, any analytic space over $\Z$ or any quasi-integral analytic space over $\F_1$. Then:
\begin{enumerate}\item $X^\mathrm{sur}$ is compactly generated and T1;
\item if $X$ is overconvergent-locally compact (def. \ref{SEP_SUR_COVER}), then $X^\mathrm{sur}$ is locally compact.\end{enumerate}
If $X$ is moreover purely analytic and quasi-separated, then $X^\mathrm{sur}$ is a universal Hausdorff quotient of $X$.\end{thm}
\begin{proof}\emph{Compact generation}. A rigid analytic space is topologically a colimit of quasi-compact open subsets, and overconvergence is detected on each subset. We will see shortly that in fact, affine subsets become compact Hausdorff.

\emph{T1.} A space is T1 if and only if points are closed, so we have to prove that points of $X$ related by specialisation are topologically indistinguishable in $X^\mathrm{sur}$. Let $x\in\overline{\{y\}}$. Then every open neighbourhood of $x$ contains $y$. In particular, there is a qcqs open set $X_0$ containing both points. Now let $U\subseteq X$ be an overconvergent open neighbourhood of $Y$. By pulling back to $X_0$, we may assume the ambient space is quasi-separated. There is a qcqs neighbourhood of $y$ contained in $U$. By corollary \ref{SEP_SUR_OPEN}, its closure, and in particular $x$, is contained in $U$.

\emph{Locally compact}. Clear from the definition.

\emph{Hausdorff.} This is \cite[I.2.5.8]{FujiKato}, which we may apply because our purely analytic space is `valuative'. For the universal property, let $q:X\rightarrow K$ be any continuous map into a Hausdorff space. It will suffice to show that $q^{-1}$ takes open sets of $K$ to overconvergent sets of $X$. We may assume $X$ is affine and hence that $K$ is quasi-compact. Let $V\subseteq K$ be open, $U\subseteq q^{-1}V$ quasi-compact. Then $qU\subseteq K$ is compact and hence closed. Therefore the embedded closure of $U$ is contained in $q^{-1}V$. Thus $q^{-1}V$ is overconvergent.
\end{proof}


\section{From rigid spaces to affine manifolds}\label{AFF}

In this section we will be interested in normal rigid analytic spaces locally of finite type over a valuation $\F_1$-field $K=\F_1(\!(t^{-H})\!)$, with ring of integers $\sh O_K=\F_1[\![t^{-H}]\!]$, for $H\subseteq\R$ an additive subgroup of the reals.



\subsection{Convergence polyhedron}

In the opening sentences \S\ref{POLY}, we made some intuitive remarks about strongly convex polyhedra defined by inequalities over $H$ inside an $H$-affine space $N$. These assemble to form a category $\mathbf{Poly}^N_H$ of embedded polyhedra, with morphisms affine maps of the ambient affine spaces that preserve the polyhedra. Our convention will be that $\emptyset$ is \emph{not} a polyhedron.

An object $\Delta$ of $\mathbf{Poly}_H^N$ is determined by the pair
\[ \Aff^+_\Delta(N,H)\subseteq\Aff_\Delta(N,H), \]
which itself determines a Banach $K$-algebra
\[ \F_1(\!(t^{-H})\!)\rightarrow \sh O\{\Delta\}=\left( \F_1\left\{z^{\Aff_\Delta(N,H)}\right\}; \F_1\left\{z^{\Aff^+_\Delta(N,H)}\right\}\right) \]
defined, as usual, by writing $\Aff^+_\Delta(N,H)\subseteq\Aff_\Delta(N,H)$ multiplicatively, adjoining $0$, and equipping it with the $t$-adic topology. It is automatically normal (since $\Aff_\Delta$ is saturated) and of finite type over $K$.

This construction is natural in $\Delta$ and hence determines a fully faithful functor
\[X:\mathbf{Poly}^N_H\rightarrow\mathbf{Rig}_{\F_1(\!(t^{-H})\!)}^\mathrm{aff/tf/n/nb},\quad \Delta\mapsto X_\Delta \]
into the category of affine and normal rigid analytic spaces of finite type over $\F_1(\!(t^{-H})\!)$ with non-boundary morphisms (that is, whose dual $\F_1$-algebra homomorphisms have no kernel).

\

Conversely, a finitely presented, quasi-integral Banach $K$-algebra $A$ gives rise to a finite-dimensional affine space
\[ N_{A/H}(-)=\Hom_H(K_A^\times,-), \]
which we consider as a functor on rank one extensions of $H$. Adapting our previous practice, write $\log f$ for the affine function on $N$ determined by $f\in K_A^\times$. Then
\[ \Delta_{A/H}=(\log f\leq 0|f\in A^+\setminus 0)\subseteq N_{A/H} \]
is a strongly convex polyhedron, the \emph{convergence polyhedron} of $A$ (cf. \cite[3.1.4]{Kap}). It depends only on the relative normalisation of $(A;A^+)$, and hence descends to a functor
\[ \Delta:\mathbf{Rig}_{\F_1(\!(t^{-H})\!)}^\mathrm{aff/tf/qi/nb}\rightarrow\mathbf{Poly}^N_H,\quad \Spec A\mapsto \Delta_{A/H} \]
right adjoint to $X$.

\begin{lemma}The adjunction $\mathbf{Poly}^N_H\leftrightarrows\mathbf{Rig}_{\F_1(\!(t^{-H})\!)}^\mathrm{aff/tf/qi/nb}$ restricts to an equivalence on the full subcategory of $\mathbf{Rig}_{\F_1(\!(t^{-H})\!)}^\mathrm{aff/tf/qi/nb}$ whose objects are normal with no torsion in $K^\times$.\end{lemma}

\begin{remark}\label{torsion}We can fix the issue with torsion in $K^\times$ by considering our polyhedra equipped with an Abelian group $K^\times$ that pairs with $N$; to avoid unnecessary complication, I will instead make the hypothesis that $K^\times$ is torsion-free a standing hypothesis from now on.\end{remark}

\begin{eg}[Field extensions]\label{AFF_EXT}If $H^\prime$ is a degree $n$ extension of $H$, then $\Delta_{H^\prime/H}$ can be realised as a $0$-dimensional polyhedron inside the $n$-dimensional $H$-affine space $H^\prime$ whose vertex is a generator of $H^\prime$ over $H$. \end{eg}

\paragraph{Open sets}
In a completed localisation of Banach $K$-algebras (cf. \S\ref{RIG_AFFINE}), we are allowed to invert some elements $s\in A\setminus 0$ and then enlarge $A^+$ by throwing in some elements of the form $t/s\in A[s^{-1}]$. This corresponds to passing to the \emph{sub-polyhedron} of $\Delta_{A/H}$ defined by the inequalities $\log t\leq \log s$. 

Thus affine open subsets of $\Spec A$ correspond to sub-polyhedra of $\Delta_{A/H}$.

\paragraph{Points}
If $H^\prime$ is an extension of $H$, then a non-boundary $\F_1(\!(t^{-H^\prime})\!)$-point of $\Spec A$ over $K$ is a commuting diagram
\[\xymatrix{ 
H^\circ\ar[r]\ar[d] & \Aff^+_{\Delta_{A/H}}(N_{A/H},H)\ar[r]\ar[d] & (H^\prime)^\circ \ar[d] \\
H\ar[r] & \Aff_{\Delta_{A/H}}(N_{A/H},H) \ar[r] & H^\prime }\]
where the top and bottom horizontal compositions are the structural map $H\subseteq H^\prime$; in other words, it is an element of $\Delta_{A/H}(H^\prime)$.

These points can actually be realised as morphisms of polyhedra; see example \ref{AFF_EXT}.

\paragraph{Convergence region}
The passage from $\Delta_{A/H}$ to $N_{A/H}$ forgets the topology (and ring of integers) of $A$: there is a Cartesian square
\[\xymatrix{
\Delta_{A/H}(-) \ar[r]\ar[d] & \Hom_K\left(A,\F_1(\!(t^{-(-)})\!) \right) \ar[d] \\
N_{A/H}(-) \ar[r] & \Hom_{K^?}\left(A^?,\F_1(\!(t^{-(-)})\!) \right)
}\]
As all functions extend meromorphically over arbitrary expansions of $\Delta_{A/H}$, the term `convergence' here is purely in analogy with the case of analysis over topological fields.

\begin{remark}Over an ordinary non-Archimedean field, any boundaryless rigid analytic space will be either an algebraic field extension or of infinite type. The existence of geometrically interesting boundaryless rigid spaces of finite type over a field is therefore a peculiarity of the $\F_1$-world.\end{remark}

In the boundaryless case, the points calculation simplifies to
\[ \Delta_{A/H}(H^\prime)\cong \Hom_K\left(A,\F_1(\!(t^{-H^\prime})\!) \right) \]
that is, $\Delta_{A/H}$ is, as a functor, simply the restriction of $\Spec A$ to the category of rank one extensions of $H$. Of course, it is possible to give a combinatorial description of the boundary as well (\S\ref{POLY_BOUNDARY}), but this is hardly more straightforward than the definition of $\Spec A$ as a functor on Banach $K$-algebras.

\subsection{Formal models}

Let $X=\Spec A\in\mathbf{Rig}_K^\mathrm{aff/tf/qi/nb}$ and let $X^+$ be a relatively normal formal model of $X$. By definition,
\[ \Hom_K(\Delta_{H^\prime}\rigless\{0\},X)=\Hom_{\sh O_K}(\Delta_{H^\prime},X^+) \]
with $\Delta_H=\Spec\F_1[\![t^{-H^\prime}]\!]$ the formal disc with exponent group $H^\prime\supseteq H$. If $\sh O_{X^+}$ is $t$-torsion-free, it is in particular quasi-integral and so we can define the punctured cone complex $\Sigma_{X^+}$ and its developing map $\Sigma_{X^+}\rightarrow N_{X^+}=\Hom(K_X^\times,-)$. The above identification yields get affine inclusions
\[\xymatrix{ \Delta_{X/H}\ar[d] \ar@{^{(}->}[r] & \Sigma_{X^+} \ar[d] \\
 N_{X/H} \ar@{^{(}->}[r] & N_{X^+}
}\]
as the fibre over the identity of the restriction map $\Hom(K_X^\times,H)\rightarrow\End(H)$. (If $H\not\subseteq\Q$, the objects in the right column may be replaced with their relative variants discussed at the end of \S I.\zref{I_PFAN_CONE}.)

If we pick $X^+=\Spec A^+$ the canonical affine model of $X$, $\Sigma_{X^+}$ will simply be the cone over $\Delta_{X/H}\subseteq\Hom(K_X^\times,-)$, punctured along the kernel of $\Hom(K_X^\times,-)\rightarrow\Hom(H,-)$. In general it will be a finite punctured fan whose support is this cone. 


Intersecting $\Sigma_{X^+}$ with $\Delta_{X/H}$ decomposes it into $H$-rational convex bodies. By compactness, any such decomposition must in fact be into finitely many $H$-rational polyhedra.

\begin{prop}The category of relatively normal models of $X$ is equivalent to the poset of polyhedral decompositions of $\Delta_{X/H}$, ordered by refinement.\end{prop}

Although we already observed this through algebra, this gives a geometric proof that:

\begin{cor}The convergence polyhedron functor $\Delta$ sends non-empty open immersions to inclusions of polyhedra.\end{cor}


\subsection{Collages}\label{AFF_COLL}

\paragraph{Glueing}The functor $\mathbf{Rig}_K^\mathrm{aff/tf/qi/nb}\rightarrow\mathbf{Poly}^N_H$ is left exact and creates limits, and therefore flat. Unlike previously, however, we now have non-trivial coverings in $\mathbf{Rig}_K^\mathrm{aff/tf/qi/nb}$, so we will need to introduce some compatibility in order to globalise our constructions.

Fortunately, it is possible to understand these coverings purely in terms of the points valued in the maximal totally ramified extension $K^\mathrm{ram}=\F_1(\!(t^{-\Q H})\!)$ of $K$ (here $\Q H\subseteq \R$ denotes the divisible hull of $H$).

\begin{lemma}\label{AFF_COVERING}A finite family of affine open subsets $U_i\subseteq X$ is a covering if and only if $X(K^\mathrm{ram})=\bigcup_i U_i(K^\mathrm{ram})$.\end{lemma}
\begin{proof}We may find a model of $X^+$ of $X$ on which each $U_i\hookrightarrow X$ is realised as an open immersion. The covering condition then becomes that $\Sigma_{X^+}$ is a union of cones in the subfans $\Sigma_{U_i^+}$; this is detected by the rational points $\Sigma_-(\Q H)$. (This would be false if we allowed an infinite family of $U_i$.)\end{proof}

Defining coverings on $\mathbf{Poly}^N_H$ to be those finite families of sub-polyhedra that induce a surjection on $\Q H$-rational points, we obtain a sheaf topos $\mathrm{Sh}\mathbf{Poly}^N_H$ and full subcategory $\mathrm{C}\mathbf{Poly}^N_H$ of locally representable objects.

\begin{defn}An object of $\mathrm{C}\mathbf{Poly}^N_H$ is called a \emph{collage in embedded $H$-rational polyhedra}, or simply \emph{collage} if the constituent objects are understood.\end{defn}

\begin{prop}The convergence polyhedron functor $\Delta$ extends to the pullback along a geometric morphism
\[ \Sh\mathbf{Poly}^N_H \rightarrow \Sh\mathbf{Rig}^\mathrm{ltf/qi/nb}_{\F_1(\!(t^{-H})\!)} \]
which preserves open immersions and induces bijections on open subset lattices.

It has a fully faithful left adjoint with image the full subcategory generated under colimits by the normal analytic spaces.\end{prop}

\begin{cor}\label{AFF_COLL_THM}The convergence region functor and its left adjoint restrict to an equivalence
\[ \Delta:\mathbf{Rig}_{\F_1(\!(t^{-H})\!)}^\mathrm{ltf/n/nb}\widetilde\rightarrow\mathrm{C}\mathbf{Poly}^N_H \]
between the category of normal analytic spaces and the category of collages.

A family of open subsets $U_i\subseteq X$ is a covering if and only if on every polyhedron $\Delta$ of $\Delta_{X/H}$ there is a finite refinement such that $\Delta(\Q H)=\bigcup_i\Delta\cap\Delta_{U_i/H}(\Q H)$.\end{cor}


\begin{eg}[Affine space]\label{AFF_AFFINE_SPACE}An $H$-affine space $N$ can be considered as a collage
\[ N(\Delta)=\Hom(\Delta(H_\infty),N(H)) \]
(which is empty unless $\Delta(H_\infty)=\Delta(H)$ is bounded). Given a strongly convex cone $\sigma$ in $\Lambda_{N/H}$, one can also define a partial compactification by allowing morphisms from infinite polyhedra whose recession cone is contained in $\sigma$.\end{eg}

\paragraph{Developing map} 
A collage in bounded polyhedra comes equipped with a locally (i.e. on each polyhedron) defined \emph{developing morphism}
\[ \delta:\Delta\rightarrow N \]
which extends globally on any universal cover. A collage in possibly unbouded polyhedra still has local developing morphisms into varying partial compactifications of $N$.

The topological realisation $\Delta(\R)$ of a locally compact collage $\Delta$ comes equipped with a local system of real affine spaces $N(\R)$ with $H$-structures and a canonical section \[ \Delta(\R) \rightarrow N(\R) \]
rendering $N(\R)$ a real vector bundle. In nice cases (cf. prop. \ref{AFF_SUR}), the developing map $\delta:\Delta\rightarrow N_p$ at $p$ will be defined on a Euclidean open neighbourhood of $p$, but this fails in general.

If $X$ has a model $X^+$, then the developing map is obtained as the fibre over $1_{\End(H)}$ of the developing map associated to the punctured cone complexes $\Sigma_{X^+}(\R)\rightarrow\Hom_\Z(K_X^\times,\R)$. More generally, $X$ admits models quasi-compact-locally, and so $\Delta_X(\R)\rightarrow N(\R)$ is a topological filtered colimit of sections of punctured cone complexes.

\subsection{Overconvergent topology of polyhedra}\label{AFF_OVER}
The complex $\Delta_{X/H}(\R_\infty)$ gives a neat description of the \emph{overconvergent topology} of $X$: the natural map $c:\Delta_{X/H}(\R_\infty)\rightarrow X$ induces a pullback on open sets of $X$
\[ c^{-1}:\sh U_{/X}\rightarrow \sh P(\Delta_{X/H}(\R_\infty)):=\{\text{subsets of }\Delta_{X/H}(\R_\infty)\},\quad U\mapsto U(\R_\infty) \]
 that matches the overconvergent sets one-to-one with the Euclidean open subsets.

\begin{thm}[Points of the overconvergent topos]\label{AFF_OVERTHM}Let $X$ be a quasi-integral rigid space, locally of finite type over $\F_1(\!(t^{-H})\!)$. The composite of the natural map $c:\Delta_{X/H}(\R_\infty)\rightarrow X$ with the separation map $b$ induces a homeomorphism between $\Delta_{X/H}(\R_\infty)$ and $X^\mathrm{sur}$.\end{thm}

The proof of this statement occupies the rest of this section.

The inverse to our pullback $c^{-1}$ will come from its right adjoint
\[ c_*:\sh P(\Delta_X(\R_\infty))\rightarrow \sh U_{/X} \quad c_*S=\bigcup_{\Delta^\prime(\R)\subseteq S}\Delta^\prime \]
which takes a subset $S\subseteq \Delta_X(\R_\infty)$ to the colimit of all $H$-rational polyhedra whose topological realisation it contains. 

Since every Euclidean open set is a union of rational polyhedra:

\begin{lemma}Let $S\subset\Delta_X(\R_\infty)$ be Euclidean open; then $c^{-1}c_*S=S$.\end{lemma}

It remains to show that the essential image of the restriction of $c^{-1}$ to $\sh U^\mathrm{sur}_{/X}$ consists of Euclidean open sets.

Applying proposition I.\zref{I_PFAN_MODIFICATION} to any model of $X$ immediately yields a characterisation:

\begin{prop}\label{AFF_MODIFICATION}Let $X$ be an affine, quasi-integral rigid analytic space of finite type over $K$, $U\subseteq V\subseteq X$ quasi-compact open subsets. Then $V$ is an overconvergent neighbourhood of $U$ if and only if $\Delta_{V/H}(\R_\infty)$ is a neighbourhood of $\Delta_{U/H}(\R_\infty)$ in $\Delta_{X/H}(\R_\infty)$.\end{prop}

Combining this and corollary \ref{SEP_SUR_OPEN}:

\begin{cor}\label{AFF_OVERCOR}An open subset $U\subseteq X$ is overconvergent if and only if $\Delta_{U/H}(\R_\infty)$ is open in $\Delta_{X/H}(\R_\infty)$.\end{cor}

This completes the proof of theorem \ref{AFF_OVERTHM} for affine $X$. The global statement is a straightforward generalisation: there is an adjunction
\[ c^{-1}:\sh U_{/X}\leftrightarrows \sh P(\Delta_X(\R_\infty)):c_* \]
with $c^{-1}$ left exact, and the overconvergent sets (resp. open sets) on the left (resp. right) are exactly those that remain so after pullback to affine set (resp. polyhedron). Therefore $c^{-1}\dashv c_*$ restrict to an equivalence $\sh U^\mathrm{sur}_{/X}\cong \sh U_{\Delta_X(\R)}$.

\subsection{Overconvergence criteria}
The arguments of this section are largely modelled on \S I.\zref{I_PFAN_CRITERIUM} - strangely, passing to analytic geometry actually introduces simplifications.

\begin{lemma}Let $\Delta$ be an overconvergent-locally compact collage. Then $\Delta(\R)$ is a locally contractible topological space.\end{lemma}
\begin{proof}The statement is straightforward in the case $\Delta$ is qcqs. More generally, if $\Delta$ is locally compact, then every polyhedron $\Delta_U$ has a qcqs overconvergent neighbourhood $\Delta_V$. By proposition \ref{AFF_MODIFICATION}, $\Delta_V(\R)$ restricts to a Euclidean neighbourhood of $\Delta_U(\R)$ on each polyhedron of $\Delta$. Since $\Delta(\R)$ is strongly topologised by polyhedra, the result follows.\end{proof}

\begin{remark}The converse is false, because in certain non-locally-compact geometries, $\Delta(\R)$ may fail to detect topological features. For instance, let $N$ be a 2-dimensional $\Z$-affine space, considered as a non-compact collage as in example \ref{AFF_AFFINE_SPACE}, and let $P_i\subset N$ be a filtered family of rational polyhedra such that
\begin{itemize}\item each $P_i$ has non-empty interior;
\item $\bigcap_iP_i=\{x\}$ with $x$ a vertex of each $P_i$.\end{itemize}
Then
\[ U:=\bigcup_i(N\setminus\mathrm{int}(P_i)) \]
is an open subset of $N$ with the same set of $\R$-points ($\simeq\R^2$). In particular, it is separated. However, the bijection $U(\R)\rightarrow N(\R)$ is not coming from a covering that is locally finite at $x$, and so by lemma \ref{AFF_COVERING} $U\hookrightarrow N$ is not an isomorphism. In fact, $U$ is not locally simply-connected at $x$.
\end{remark}

In particular, the developing map of a locally compact collage extends to a qcqs overconvergent neighbourhood of any polyhedron. This allows us to apply the arguments of propositions I.\zref{I_PFAN_SEP} and I.\zref{I_PFAN_PROP} to prove analogous statements for analytic spaces:

\begin{prop}\label{AFF_SUR}Let $\Delta\in\mathrm{C}\Poly_H^N$ be a locally compact collage. Then the developing map $\delta:\Delta\rightarrow N$ is overconvergent-locally defined. Moreover:
\begin{enumerate}\item $\Delta$ is locally separated if and only if $\delta$ is an overconvergent-local immersion;
\item $\Delta$ is overconvergent if and only if $\delta$ is an overconvergent-local homeomorphism.\end{enumerate}\end{prop}
\begin{proof}Since the arguments are essentially the same as those of \emph{loc. cit.}, I include only one part, by way of illustration.

Suppose that $\Delta$ is overconvergent, and let $\Delta_U\subseteq\Delta$ be a polyhedron. There is a qcqs overconvergent neighbourhood $\Delta_U^\prime$ of $\Delta_U$ such that $\delta:\Delta_U^\prime\rightarrow N_U$ is defined. Let $\Delta_V\subseteq N_U$ be a polyhedron whose real points contain a neighbourhood of $\Delta_U(\R_\infty)$. By overconvergence, after possibly shrinking $\Delta_V$, there is a unique section
\[ \Delta_V\rightarrow\Delta_U^\prime \] to $\delta$. Thus $\delta$ is a local homeomorphism.\end{proof}

For simplicity, I have stated the absolute version, though to derive the relative version as in I.\zref{I_PFAN_PROP} would be straightforward. Note that unlike the case of punctured cone complexes, the collages appearing in this theorem are not necessarily quasi-separated.

\begin{cor}An $H$-affine space $N$, considered as a collage (e.g. \ref{AFF_AFFINE_SPACE}), is overconvergent. The natural map $N(\R)\rightarrow N^\mathrm{sur}$ is a homeomorphism.\end{cor}

It follows from theorem \ref{AFF_OVERTHM} that in fact:

\begin{cor}Let $\Delta\in\mathrm{C}\Poly_H^N$ be a locally compact collage in bounded polyhedra. Then the developing map $\delta:\Delta(\R)\rightarrow N(\R)$ is locally defined. Moreover:
\begin{enumerate}\item $\Delta$ is locally separated if and only if $\delta$ is an local immersion;
\item $\Delta$ is overconvergent if and only if $\delta$ is an local homeomorphism.\end{enumerate}\end{cor}

\begin{cor}\label{AFF_SEPARATED}Let $X$ be a quasi-integral analytic space locally compact and locally of finite type over $\F_1(\!(t^{-H})\!)$. The developing map \[ \delta:\Delta_{X/H}\rightarrow N_{X/H}\] is locally defined. Moreover:
\begin{enumerate}\item $X$ is locally separated if and only if $\delta$ is a local immersion;
\item $X$ is overconvergent if and only if $\delta$ is a local homeomorphism.\end{enumerate}\end{cor}

\begin{cor}Let $X$ be overconvergent. Then $\Delta_X(\R)$ carries a unique structure of an affine manifold with developing map $\delta$.\end{cor}

\subsection{$H$-affine manifolds as rigid analytic spaces over $\F_1(\!(t^{-H})\!)$}
Here we globalise the constructions of \S\ref{AFF_OVER}. For simplicity, we will treat only the \emph{boundaryless} rigid analytic spaces. One can classify more general normal analytic spaces in terms of affine manifolds with corners; a sketch-definition of such objects can be found in \S\ref{POLY_AFF}.

\begin{defn}A rigid analytic space is said to be \emph{boundaryless} if its only closed subspaces are unions of connected components. Equivalently, it is locally modelled by the spectra of locally convex $\F_1$-fields.\end{defn}

Note that the condition of being boundaryless in particular implies normality. It does not say anything about the topological boundary of $\Delta(\R)$. 

Let $\mathbf{Aff}_H$ be the category of paracompact $H$-affine manifolds and $H$-affine maps. It is finitely complete, and generated under colimits by open subsets of $H$-affine space. The Euclidean topology endows $\mathbf{Aff}_H$ with the structure of a spatial, but not coherent, site. If we allow affine manifolds to be non-Hausdorff, then Yoneda matches $\mathbf{Aff}_H$ with the category of all paracompact locally representable objects of $\Sh\mathbf{Aff}_H$.

Every object of $\mathbf{Aff}_H$ is exhausted by bounded polyhedra, and so the functor
\[ c^{-1}:\mathbf{Aff}_H\rightarrow\Sh\mathbf{Poly}_H^N, \quad U\mapsto\left[ \Delta\mapsto \left\{\begin{array}{cl}  \Hom_H(\Delta,U) & \Delta\text{ bounded} \\ \emptyset & \text{otherwise}\end{array}\right.\right]   \]
is fully faithful.

\begin{lemma}$c^{-1}$ preserves finite limits, open immersions, and coverings.\end{lemma}
\begin{proof}\emph{Lex.} An $H$-affine map $\Delta\rightarrow U$ is determined by its underlying map of sets $\Delta(H)\rightarrow U(H)$, and affineness is stable for fibre products. Therefore $c^{-1}$ is left exact.

\emph{Open immersions.} Clear from the definition.

\emph{Coverings.} Let $U=\bigcup_iU_i$ be an open covering. Since $U$ is paracompact, the cover may be assumed locally finite. It therefore restricts to a locally finite family on each polyhedron that is covering for $\R$-points, and hence a covering by lemma \ref{AFF_COVERING}.\end{proof}

It therefore extends to a geometric morphism
\[ c:\Sh\Poly_H\rightarrow \Sh\mathbf{Aff}_H,\qquad c^*:\mathbf{Aff}_H\rightarrow\mathrm{C}\mathbf{Poly}_H \]
whose pullback preserves locally representable objects. It also respects developing maps. By the criterion of proposition \ref{AFF_SUR}, it follows that the objects in the essential image of $c^*$ are actually \emph{overconvergent}. By construction, they also do not admit morphisms from any collage with boundary, and so are themselves boundaryless.

\begin{thm}The geometric morphism \[c:\Sh\Poly_H^{\not\partial/\mathrm{sur}} \rightarrow \Sh\mathbf{Aff}_H\] is a spatial equivalence of categories.\end{thm}
\begin{proof}We have already seen that $c^*$ is fully faithful. It remains to show that the image of $\mathbf{Aff}_H$ is a site for the overconvergent topos. This is a consequence of corollary \ref{AFF_SUR} and the fact that $\mathbf{Aff}_H$ generates the overconvergent topology of any $H$-affine space $N$.\end{proof}

By composing all our functors, we obtain the fully faithful functor
\[ \mathbf{Aff}_H\rightarrow\mathbf{Rig}_K^{\not\partial/\mathrm{sur}} \]
that is the title of this document. The objects in the essential image of this functor are paracompact, quasi-separated, overconvergent, and boundaryless (which for obvious reasons, I have chosen not to record in the superscript). 

It realises an affine manifold $B$ as a rigid analytic space $B^\mathrm{rig}$, together with a (discontinuous) map $c:B\rightarrow B^\mathrm{rig}$ that descends to a homeomorphism on the overconvergent site.

\begin{cor}\label{AFF_AFF}The restriction of $c^*$ is a topological equivalence
\[ \mathbf{Aff}_H \cong \mathbf{Rig}_{\F_1(\!(t^{-H})\!)}^{\not\partial/\mathrm{sur}} \]  
between the category of (not necessarily Hausdorff or paracompact) $H$-affine manifolds and the category of boundaryless, overconvergent, rigid analytic spaces. 

Moreover, the following are equivalent:
\begin{enumerate}\item $X$ has affine diagonal;
\item $X$ is quasi-separated;
\item $\Delta_X(\R)$ is Hausdorff.\end{enumerate}\end{cor}
\begin{proof}Suppose that $\Delta(\R)$ is Hausdorff, and let $\Delta_1,\Delta_2\subseteq\Delta$ be two polyhedra. By the Hausdorff property, $\Delta_i(\R)$ is closed in $\Delta(\R)$.

There are overconvergent neighbourhoods $U_i$ of $\Delta_i$ such that $\delta:U_i\hookrightarrow N(\R)$ is an open immersion. The intersection $\Delta_2\cap U_1$ is closed in $U_1$ and contains $\Delta_1\cap\Delta_2$. It follows that the latter can be realised as an intersection of two polyhedra inside an affine space, and is therefore a polyhedron. This proves that $\Delta$ has affine diagonal.\end{proof}

\begin{eg}[Groups]\label{AFF_GROUPS}The inclusion $\mathbf{Rig}_K^\mathrm{ltf/n/nb}\hookrightarrow\mathbf{Rig}_K$ being left exact, it follows that an affine manifold with the structure of a \emph{group} induces a group structure on its rigid analytic space. Many motivating examples of toric rigid analytic spaces arise in this way.

For example, the rigid space associated to the affine manifold $H\subseteq\R$ is a group object $\G_m/K$ of $\mathbf{Rig}_K$, the \emph{multiplicative group} over $K$. More generally, to the affine manifold $\Hom(\Lambda,H)$, where $\Lambda$ is any lattice, one puts $\Hom(\Lambda,\G_m)$ for the \emph{diagonalisable group with character group $\Lambda$}. Note that in contrast to the algebraic setting, these diagonalisable groups are not quasi-compact.\end{eg}

\subsection{On the Mumford degeneration}Continuing on from the previous example, let us investigate proper, commutative group objects of $\mathbf{Rig}_{\F_1(\!(t^{-H})\!)}$. Under the correspondence of corollary \ref{AFF_AFF}, these are nothing more than compact $H$-affine group manifolds. 

Let $B$ be an $H$-affine group manifold. The model affine space $N_B$ at the origin is naturally a vector space. We will assume that $B$ is connected and complete, that is, that the developing map $\widetilde B\rightarrow N_B$ is an isomorphism. (The latter is probably automatic for compact affine manifolds with integer slopes.) 

Note that here, as above (remark \ref{torsion}), we are also implicitly restricting attention to the case where $K_B^\times$ is torsion-free. In geometric terms, this means that a base change to a field of characteristic zero is connected. Let us call this property \emph{geometrically connected}.

In this case, the universal covering gives a uniformsation
\[  B\cong N_B/Y \]
with $Y=\pi_1(B,0)$ a cocompact, discrete subgroup of $N_B(H)$; that is, a lattice. In particular, $B$ is an affine torus with circumferences in $H$.

This uniformisation translates onto the rigid analytic side; the covering $N_B$ corresponds, as in example \ref{AFF_GROUPS}, to a torus $T=N_B\tens\G_m$ with character group $X^*(T)=\Lambda_{B/H}$, and $Y$ to a subgroup of the $\F_1(\!(t^{-H})\!)$-points of $T$. We get the following commutative square of equivalent categories:
 \[\xymatrix{  {\left\{\begin{matrix}\text{proper, geometrically connected,} \\ \text{commutative groups in }\mathbf{Rig}_{\F_1(\!(t^{-H})\!)} \end{matrix} \right\}} \ar@{-}[r]^-{\text{cor. }\ref{AFF_AFF}}  \ar@{-}[d] & 
\left\{H\text{-affine tori }\right\}  \ar@{-}[d] \\
 {\left\{\begin{matrix}\text{ pairs $(T,Y)$ with $T$ an algebraic torus over} \\ \text{ $\F_1(\!(t^{-H})\!)$ and $Y\subseteq T(\F_1(\!(t^{-H})\!))$ a lattice with} \\ \text{$\mathrm{rk}(Y)=\mathrm{rk}(T)$ and $\sh O(T)\rightarrow\sh O(Y)$ injective}\end{matrix}\right\}} \ar@{-}[r] &
{\left\{\begin{matrix}\text{ pairs $(N,Y)$ of lattices with} \\ \text{$Y\subseteq N\tens H$ inducing an} \\ \text{isomorphism $Y\tens H\cong N\tens H$}\end{matrix}\right\}}
}\]

The intuition behind this result is essentially the same as that of Mumford's seminal paper \cite{Mumford}, but expressed in a somewhat different language. Here I include a few paragraphs by way of translation; my notation follows that of \cite{FaltingsChai}.

Let $K$ be a non-Archimedean field, $T$ a split algebraic torus of rank $n$ over $K$. We also fix a free $\Z$-module $Y\simeq\Z^n$ of the same rank. In \cite{FaltingsChai}, the authors distinguish three additional sets of data defining the completely degenerate Abelian variety $B_K:=T/Y$:
\begin{enumerate}
\item a pairing $b:X^*(T)\tens Y\rightarrow K^\times$;
\item a homomorphism $\phi:Y\rightarrow X^*(T)$;
\item a `quadratic' function $a:Y\rightarrow K^\times$.\end{enumerate}
In more geometric terms, these data correspond to:
\begin{enumerate}\item the `period' lattice $b^\tau:Y\hookrightarrow T$;
\item a cocycle $\phi$ for a $\theta$-line bundle on $T/Y$;
\item a metric on $L$.\end{enumerate}
We have already seen the lattice $Y$ appearing in our $\F_1$-construction. Let us choose a section $H\rightarrow K^\times$ of the valuation, and suppose that $b$ takes values in the image of this section. Then the analytic space $B_K$ is obtained by base change from the $\F_1(\!(t^{-H})\!)$-analytic space associated to an affine torus $B=(X_*(T)\tens H)/Y$ along the induced map $\F_1(\!(t^{-H})\!)\rightarrow K$, and we can make the identification $X^*(T)\cong\Lambda_{B/H}$.

I would like to argue that the $\theta$-bundle is also defined at this level. Indeed, the very fact that the cocycle $\phi$ is takes values in the character group, rather than arbitrary Laurent polynomials, implies in particular that $L$ is a monomial line bundle. More precisely, $\phi$ defines an element of
\[ \Pic\left(B/{\F_1(\!(t^{-H})\!)}\right) \quad \cong \quad H^1(B,\Aff(B,H)) \quad \cong\quad\Hom(Y,H)\oplus \Hom(Y,\Lambda_{B/H}), \]
where we use the origin of $N$ to split the cotangent sequence $\Aff(N,H)\cong H\oplus \Lambda_{B/H}$. The fact that it even appears an element of the second factor $\Hom(Y,\Lambda_{B/H})\cong\Hom(Y,X^*(T))$ implies that it is \emph{metrisable}. I do not discuss the metric itself here.

I conclude this section by highlighting how the `geometric' perspective of this paper stands up to the full generality of \cite{FaltingsChai}:
\begin{itemize}\item Using the technology presented in this paper as written, it is only possible to get Abelian varieties with `monomial structure constants', that is, such that $Y$ acts on $T$ with co-ordinates powers of the uniformiser. It is not difficult to reintroduce general $K^\times$ structure constants into the mix, but we defer that to a later project.
\item The theta series cannot be defined directly over $\F_1(\!(t)\!)$ since they involve addition. In other words, the $\theta$-bundles have no sections until they are pulled back to a base where addition is defined.
\item One can also study families over a higher-dimensional base, but one will lose the interpretation in terms of affine manifolds.
\item The affine manifolds perspective does not, it seems, offer any insight into the Abelian part of the Raynaud extension.
\end{itemize}

\subsection{Rigid analytic spaces as twisted affine manifolds}

One can also cook up an object that represents the overconvergent site of a purely analytic rigid space even without a trivialisation of the normal bundle to the punctures. The following discussion is informal.

Let $X$ be paracompact, and pick a relatively normal model $X^+$. Then
\[ \Sigma_{X^+}(H)=X\left(\F_1(\!(t^{-H})\!)\right) \]
is a fan. Without the base field $K$ to anchor it, the set of $H$-points has a scaling degree of freedom. The topological realisation $\Sigma_{X^+}(\R)$ depends only on $X$, so we write simply $\Sigma_X(\R)$.

As manifolds,
\[ X^\mathrm{sur}\simeq \Sigma_{X^+}(\R)/\R^\times_{>0} \]
so $\Sigma_X(\R)$ can be considered as the bundle of positive rays in an \emph{oriented real line bundle} $\ell^\vee$ on $X^\mathrm{sur}$. The fibre of $\ell$ over any rational (resp. integral) point of $X^\mathrm{sur}$ comes with a set of rational (resp. integral) points respected by the structure group.

The locally constant sheaf $K_X^\times$ can no longer be considered as a subsheaf of $C^0(X^\mathrm{sur},\R)$. Rather, it is a sheaf of functions
\[ \ell^\vee\rightarrow \R\]
respecting the integral structures, that is, of \emph{affine sections} of the dual line bundle $\ell$. The zero section of $\ell$ is the global element $1\in\Gamma K_X^\times$, and a non-vanishing global affine section of $\ell$ corresponds to a choice of structural morphism $X\rightarrow\Spec\F_1(\!(t^{-\Q})\!)$.

\begin{thm}Let $X$ be a paracompact, boundaryless, purely analytic, locally Noetherian rigid analytic space over $\F_1$. Suppose that $X$ is formally overconvergent over $\F_1$.

Then $X^\mathrm{sur}$ is a manifold, and there is a canonical oriented real line bundle $\ell$ on $X^\mathrm{sur}$ and embedding $K_X^\times\hookrightarrow \Gamma(\ell/X)$ such that locally, these structures are isomorphic to those on an affine manifold with $\ell=\R$ constant. Moreover, $X$ can be recovered from the data $\ell/X^\mathrm{sur},K_X^\times\hookrightarrow\Gamma(\ell/X)$.\end{thm}

The line bundle $\ell$, including its affine structure, computes the topological type of the normal bundle to the boundary in a log smooth model of $X$. 

\begin{eg}[Hopf map]Let $S^3$ be the punctured formal completion of the origin in $\A^2$. Identifying the blow-up of the plane at the origin with $\sh O_{\P^1}(-1)$, we obtain a `Hopf fibration' 
\[ S^3\rightarrow \P^1  \]
 the fibre over any $k$-point, with $k$ an $\F_1$-field, being isomorphic to $\Spec k(\!(t)\!)$.

The punctured fan associated to the blow-up of $\hat \A^2$ is the lower quadrant in $\R^2$ divided into two cones by the negative diagonal, so $(S^3)^\mathrm{sur}$ is homeomorphic to a closed interval. Of course, both $\ell$ and $K_X^\times$ are topologically trivial. However, there is no \emph{affine} trivialisation of $\ell$: any element of $K_X^\times=\Z^2$ must annihilate some ray in $\R^2_{\leq0}$ and hence, under the embedding $K_X^\times\hookrightarrow\ell$, vanish at the corresponding point of $(S^3)^\mathrm{sur}$. 

In fact, a non-zero affine section of $\ell$ vanishes at exactly one point. This computes the Euler class, $1$, of the conormal bundle $\sh O(1)$ to $\P^1$. I leave it to the reader to imagine his own generalisations.\end{eg}


\section{De $\F_1(\!(t)\!)$ \`a $\C$}\label{C}
We already showed that there is a base change functor from the category of rigid analytic spaces over $\F_1(\!(t)\!)$ to any non-Archimedean field endowed with a topological nilpotent. Remarkably, one can also base change to \emph{Archimedean} fields: there is a family of continuous monoid homomorphisms
\[ \F_1(\!(t)\!)\rightarrow\C,\quad t\mapsto q \]
indexed by $q=e^\epsilon$ in the punctured open unit disc $\Delta^*\subset\C^\times$. The monoid $\F_1(\!(t)\!)$ does not distinguish between `Archimedean' and `non-Archimedean' topologies.

Although the formalisms are not immediately compatible, this can nonetheless be globalised to obtain a base change functor into the category of complex analytic spaces, in generalisation of the construction
\[ B\mapsto TB/\epsilon\Lambda^\vee  \]
of torus fibrations described in the introduction to part I.

Let $S$ be a set with a $\Z$-indexed increasing filtration $F$, and define a norm
\[ |\!|(c_f)_{f\in S}|\!| :=\sum_{k\in\Z}\sum_{f\in F_kS\setminus F_{k-1}S}|q^kc_f| \]
on the $\C$-vector space $\bigoplus_S\C$. Its completion is the Banach space
\[ \ell^1_q(S,F):= \left\{(c_f)_{f\in S}\in\C^S\left| |\!|(c_f)|\!|  <\infty\right.\right\}  \]
of absolutely $q$-summable $S$-indexed sequences.

If $A$ is a Banach $\F_1(\!(t)\!)$-algebra, then $A\hat\tens_q\C:=\ell^1_q(A,\mathrm{ord}_t)$ is a Banach $\C$-algebra with respect to the projective tensor product. 

\begin{lemma}Suppose $A$ is reduced and of finite type. Then $A\hat\tens_q\C$ is the ring of holomorphic functions on $\Spec A(\C)$.\end{lemma}
\begin{proof}Let us choose a surjection $\F_1[\![t]\!]\{x_i\}_{i=1}^k\twoheadrightarrow A^+$. This embeds the affine algebraic variety $\Spec A^?(\C)$ into an affine space $\A^k_\C$ in such a way that the $t$-adic norm on $A^+$ is the $\mathrm{L}^\infty$-norm on the unit polydisc of $\A^k_\C$. By the maximum principle, this is the same as the $\mathrm{L}^\infty$-norm on the unit torus. This is calculated by the formula for $|\!|-|\!|$ introduced above.

Since $A$ is reduced, the discrete $\C$-algebra $A^?\tens_q\C$ is exactly the set of polynomial functions on $\Spec A^?(\C)$. Thus $A\hat\tens_q\C$ is the completion of the space of polynomial functions for the $\mathrm{L}^\infty$-norm on the intersection $\Spec A(\C)$ of $\Spec A^?(\C)$ with the closed unit polydisc. Since this space is compact, the induced topology is the topology of uniform convergence, and hence the completion is the space of holomorphic functions.\end{proof}

Suppose that $\Spec A\subseteq\Spec B\subseteq V$ are affine open immersions, and that $\Spec B$ is an overconvergent neighbourhood of $\Spec A/V$. By finiteness of global sections over blow-ups, $B\hat\tens_q\C\rightarrow A\hat\tens_q\C$ is a nuclear operator. It follows that any non-constant holomorphic function from $\Spec B(\C)$ into the closed unit disc maps $\Spec A(\C)$ into the open unit disc, that is, the latter is contained in the topological interior of the former.

If $U\subseteq\Spec A$ is an overconvergent open subset, then $U(\C)$ is a complex analytic space without boundary. Indeed, by corollary \ref{SEP_SUR_OPEN}, overconvergence means that every compact subset is in the topological interior of a larger compact subset.

The condition of local convexity (def. \ref{SEP_COMPACT}) being overconvergent-local, we can define the site $\mathbf{Rig}_{\F_1(\!(t)\!)}^\mathrm{sur/lc}$ of locally convex, overconvergent analytic spaces over $\F_1(\!(t)\!)$. It is generated by overconvergent spaces $U$ admitting an open immersion into an affine object $\Spec A$.\footnote{For $\F_1(\!(t^{-H})\!)$, one can show with a little effort that in fact all overconvergent spaces are locally convex, but we will not make that digression here.} To each such space, we have associated a complex analytic space $U(\C)$. Following general principles, this extends to a spatial base change
\[ (-)\times_q\Spec\C:\mathbf{Rig}_{\F_1(\!(t)\!)}^\mathrm{sur/lc}\rightarrow\mathbf{An}_\C \]
to the category of (possibly non-reduced) complex analytic spaces. One therefore has in general a map
\[ \mu:X\times_q\Spec\C\rightarrow X^\mathrm{sur} \]
of spaces equipped with sheaves of topological monoids that specialises, in the case that $X$ is boundaryless, to the construction of analytic torus fibrations described in the introduction.

\appendix
\section{On the typing of points of rigid analytic spaces}\label{APPEND}
Berkovich famously classified points on analytic curves $C$ over a non-Archimedean field in to four `types'. With a small modification, it is easy to extend this to Huber's framework and thus understand all points of the Riemann-Zariski space $\mathrm{RZ}(C)$ (\S\ref{RIG_RZ}). However, the generalisation to higher dimensions is somewhat more complicated.

For analytic spaces of finite type over $\F_1(\!(t)\!)$, the situation is sufficiently simple that we can give a complete solution. The pathological type IV points do not appear.

\subsection{Types I and III components}

Let $\Delta$ be a rational polyhedron. We have constructed in \S\ref{AFF_OVER} a continuous map
\[ b:X_\Delta\rightarrow\Delta(\R_\infty) \]
where $X_\Delta$ is the rigid analytic space over $\F_1(\!(t)\!)$ associated to $\Delta$, with a discontinuous section $c$.

Let $x\in X_\Delta$. The following invariants are visible at the level of its image $bx$:

\begin{defns}
The \emph{height}, or \emph{type I codimension} or \emph{component}, of $x$ is the codimension of the (smallest) infinite face $\Delta_x$ of $\Delta$ containing $x$.

The \emph{rational rank}, or \emph{type III component}, is the dimension of the smallest rational subspace of $\Delta$ that contains $x$. Equivalently, it is the rank over $\Q$ of the linear map
\[ \Aff(\Delta_x,\R)\tens_\Z\Q \rightarrow \R \]
induced by $bx$.

If the height, resp. rational rank of $x$ equals the dimension of $\Delta$, we say that $x$ is \emph{purely of type I} (resp. \emph{III} or \emph{irrational}). Note that $\Delta$ can have at most one purely type I point: its infinite vertex.
\end{defns}

Intuitively, if $x$ has type I (resp. III) component $d_\mathrm{I}$ (resp. $d_\mathrm{III}$), it means that under a suitable co-ordinate system we may write
\[ x = (-\infty,\cdots,-\infty|q_1,\cdots,q_\ell|r_1,\cdots,r_{d_\mathrm{III}}) \]
with the first $d_\mathrm{I}$ co-ordinates equal to $-\infty$, the last $d_\mathrm{III}$ co-ordinates $r_i$ relatively irrational, and each $q_i\in\sum_{j=1}^{d_\mathrm{III}}r_j\Q$.

\subsection{Type II component}

Let $x\in\Delta$. We may eliminate the types I and III components of $x$ by replacing $\Delta$ with the quotient of $\Delta_x$ by the rational span of $x$; the image $x_\mathrm{II}$ of $x$ in this quotient is called the \emph{type II component} of $x$. If $x=x_\mathrm{II}$ we say it is or \emph{purely of type II}. For describing the image $bx$ in $\Delta(\R_\infty)$, we may also replace the morpheme `type II' with \emph{rational}.

In this section we will enumerate the additional data required to lift a point $y\in\Delta(\R_\infty)$ to $\Delta$; since these data will depend only on the component $y_\mathrm{II}$, it suffices to suppose that $y$ is purely of type II.

\begin{defn}Let $y\in\Delta(\R)$ be purely of type II. An \emph{oriented $d_\mathrm{II}$-flag at $y$} is a complete flag $\rho_\bullet$ of half-spaces of the conormal cone $\Lambda^\vee_{\Delta,y}$ of dimensions up to $d_\mathrm{II}\in\N$; that is, $\rho_1$ is a ray (pointing into $\Delta$) and $\rho_{i+1}$ is a half-space with boundary the hyperplane $\pm\rho_i$.\end{defn}

It makes sense to ask if a jet in $N(\R)$ at $y$ - in the sense of differential geometry - lifts a specified flag $\rho$. A flag $\rho_\bullet$ is contained in a given polyhedron $\Delta$ if and only if there is a jet lifting $\rho_\bullet$ that points into $\Delta(\R)$.

An oriented flag is equivalent to the data of a flag of half-spaces in the dual space $\Lambda_{\Delta,y}$ of \emph{codimensions} up to $d_\mathrm{II}$. Such a flag can be packaged in terms of the associated total order on $\Lambda/\rho^\perp_{d_\mathrm{II}}$ whose convex composition factors are $\pm\rho_i^\perp/\rho_{i+1}^\perp\cong\Z$ with the obvious ordering.

\begin{prop}The set $\sh J_y\Delta$ of oriented flags of $\Delta$ at $y$ is the set of surjective group homomorphisms
\[ v:\mathrm{Aff}_\Delta(N,\Z)\rightarrow\Gamma \]
into totally ordered groups $\Gamma$ such that $v(F)\leq 0$ when $F\leq 0$ in a neighbourhood of $y$ in $\Delta(\R)$.\end{prop}

Thus oriented flags can be thought of as higher rank valuations of $\sh O\{\Delta\}$.

The set  of oriented flags at $y_\mathrm{II}$, where $d_\mathrm{II}$ is allowed to vary, carries a natural partial order by containment of flags. The unique zero-flag is the minimum element of $\sh J_y\Delta$. We will see that this partial order is precisely the specialisation order of $b^{-1}y$.

\subsection{The set of points}

The definition of the point-set topology of $X_\Delta$ is designed so that the conclusion of Stone's theorem holds. As such, its set of points may be identified with the set of \emph{prime filters} of quasi-compact open subsets of $X_\Delta$. We may identify this set in finitary terms.

A filter $\varphi$ being prime means that if a set $U\in\varphi$ can be decomposed as the union $U_1\cup U_2$ of two subsets, then either $U_1$ or $U_2\in\varphi$. Since every quasi-compact open subset of $X_\Delta$ is the union of finitely many polyhedra, it follows:

\begin{lemma}Every prime filter of $\sh U_{/X_\Delta}^\mathrm{qc}$ is generated by polyhedra.\end{lemma}
Moreover, since one can check at the level of $\Delta(\R_\infty)$ whether a collection of polyhedra covers a given subset:
\begin{lemma}If $\varphi\in X_\Delta$, the filter of polyhedra containing $\varphi$ is prime.\end{lemma}
We have therefore reduced the problem to describing the set of prime filters of polyhedra in the Hausdorff quotient $\Delta(\R_\infty)$.

Let me introduce the temporary notation $\sh J\Delta$ for the set of all pairs $(y,(y_\mathrm{II},\rho_\bullet))$ of a point $y\in\Delta(\R_\infty)$ and an oriented flag at the rational part $y_\mathrm{II}$. The union of two sub-polyhedra contains such a pair if and only if one of the two constituents does. This produces a map from $\sh J\Delta$ to the set of prime filters of polyhedra.

\begin{prop}\label{blarg}The induced map $\sh J\Delta\rightarrow X_\Delta$ is bijective.\end{prop}
\begin{proof}This is the statement that every prime filter contains a unique oriented flag. Of course, every prime filter contains at least the length zero flag $\bigcap_{U\in\varphi}U$; the unicity reduces to the following elementary fact:

\begin{lemma}Elements of $\sh J\Delta$ are separated by polyhedra.\end{lemma}
In other words, for any two distinct flags we can find a pair of filters each containing just one of the points.\end{proof}

\begin{defn}Let $x\in X_\Delta$. The types I and III components $d_\mathrm{I},d_\mathrm{III}$ of $x$ are defined according to those of $bx$. The \emph{type II codimension} or \emph{residue height} of $x$ is the length $d_\mathrm{II}$ of the flag $\rho_\bullet x\in\sh J_{bx}\Delta$ at $bx_\mathrm{II}$ corresponding to $x$ under $X_\Delta\rightarrow\sh J\Delta$.

We say in this case that $x$ is of type $(d_\mathrm{I}|d_\mathrm{II}|d_\mathrm{III})$.\end{defn}

\subsection{Structure of the fibres of the Hausdorff quotient}

Proposition \ref{blarg} identifies the closed subspace $b^{-1}y\subseteq\mathrm{sk}\Delta$, as a set, with $\sh J_y\Delta$. Moreover, the partial order of the set of filters by inclusion restricts, under this correspondence, to the partial order of weak jets by inclusion. In other words, this is also the specialisation order of $b^{-1}y$, as promised.

The irreducible closed subsets of $\sh J_y\Delta$ are therefore of the form
\[ V_\rho:=\{ \rho^\prime | \rho\subseteq\rho^\prime \} \]
with $\rho\in\sh J_y\Delta$.


The embedding $b^{-1}y\hookrightarrow X_\Delta$ can be upgraded to an affine morphism of analytic spaces by equipping the left hand side with the pullback of the structure sheaf. With this structure, at least when $y$ is purely of type II, $\sh J_y\Delta$ is the spectrum of the pair $(A;A^+)$ whose ring of integers is the $\F_1$-algebra associated to
\[ \mathrm{Aff}^+_{\Delta,y}(N,\Z) :=  \{f\in\mathrm{Aff}_\Delta(N,\Z)| f\leq 0\text{ in a neighbourhood of }y\}. \]
Note that this is not, of course, a Banach algebra topology.

\bibliographystyle{alpha}
\bibliography{polyb}

\end{document}